\documentclass[11pt,a4paper]{article}

\usepackage[OT1]{fontenc}
\usepackage{lmodern}
\usepackage[english]{babel}
\usepackage[utf8x]{inputenc}
\usepackage[sc]{mathpazo}
\usepackage{amsmath,amssymb,amsfonts,mathrsfs,amstext}
\usepackage[thmmarks]{ntheorem}
\usepackage{wrapfig}
\usepackage{nicefrac}
\usepackage{tabularx}

\usepackage{varioref}

\usepackage{datetime}

\usepackage{mathtools}

\usepackage{array}

\usepackage{listings}
\usepackage[linkcolor=black,colorlinks=true,citecolor=black,filecolor=black]{hyperref}

\usepackage{algpseudocode}
\usepackage{algorithm}

\usepackage{color}
\usepackage{caption}
\usepackage{subcaption}
\usepackage[noabbrev]{cleveref}
\usepackage{csquotes}
\usepackage{bm}
\usepackage{mathtools}
\usepackage{booktabs}

\numberwithin{equation}{section}

\newtheorem{theorem}{Theorem}[section]

\newtheorem{remark}[theorem]{Remark}

\newtheorem{lemma}[theorem]{Lemma}

\theoremstyle{nonumberplain}
\theorembodyfont{\normalfont}
\theoremsymbol{\ensuremath{\blacksquare}}
\newtheorem{proof}{Proof}

\renewcommand{\epsilon}{\ensuremath\varepsilon}

\renewcommand{\phi}{\ensuremath{\varphi}}

\DeclareMathAlphabet{\mathpzc}{OT1}{pzc}{m}{it}

\newcommand*\Laplace{\mathop{}\!\mathbin\bigtriangleup}

\newcommand{\NORM}[1]{\left\lVert#1\right\rVert} 

\newcommand{\DEF}{\coloneqq}

\newcommand{\RR}{\mathbb{R}}
\newcommand{\CC}{\mathbb{C}}
\newcommand{\NN}{\mathbb{N}}
\newcommand{\ZZ}{\mathbb{Z}}
\newcommand{\Om}{\Omega}

\newcommand{\del}{\partial}

\newcommand{\MID}{\!\! \mid\!}

\newcommand{\cC}{\mathcal{C}}
\newcommand{\eps}{\epsilon}

\newcommand{\OO}{\mathcal{O}}

\newcommand{\udk}{u^{\delta k}}
\newcommand{\udknull}{u^{\delta k}_0}

\newcommand{\udks}{u^{\delta k}_s}

\newcommand{\Uk}{U^{k}}
\newcommand{\Uknull}{U^{k}_0}

\newcommand{\Uks}{U^{k}_s}

\newcommand{\Gkp}{\Gamma^k_+}

\newcommand{\Gdkp}{\Gamma^{\delta k}_+}

\newcommand{\intd}{\mathrm{d}}

\newcommand{\curlXe}{{\mathcal{X}^\epsilon}}
\newcommand{\curlYe}{{\mathcal{Y}^\epsilon}}

\newcommand{\LLe}{\mathcal{L}^\eps}
\newcommand{\LLei}{\mathcal{L}^{\eps_i}}
\newcommand{\LLeinv}{(\LLe)^{-1}}

\newcommand{\TransT}{\mathrm{T}}

\title{{A mathematical and numerical framework for gradient meta-surfaces built upon periodically repeating arrays of Helmholtz resonators}}
\date{}
\author{ Habib Ammari\thanks{\footnotesize Department of Mathematics, ETH Z\"urich, R\"amistrasse 101, CH-8092 Z\"urich, Switzerland (habib.ammari@math.ethz.ch, kthim.imeri@sam.math.ethz.ch).} 
\and Kthim Imeri\footnotemark[1] 
}

\begin{document}
	\maketitle

\begin{abstract}
In this paper a mathematical model is given for the scattering of an incident wave from a surface covered with microscopic small Helmholtz resonators, which are cavities with small openings. More precisely, the surface is built upon a finite number of Helmholtz resonators in a unit cell and that unit cell is repeated periodically. To solve the scattering problem, the mathematical framework elaborated in [Ammari et al., Asympt. Anal., 114 (2019), 129--179] is used.
The main result is an approximate formula for the scattered wave in terms of the lengths of the openings. Our framework provides analytic expressions for the scattering wave vector and angle and the phase-shift. It justifies the apparent absorption. Moreover, it shows that at specific lengths for the openings and a specific frequency there is an abrupt shift of the phase of the scattered wave due to the subwavelength resonances of the Helmholtz resonators.  A numerically fast implementation is given to identify a region of those specific values of the openings and the frequencies.
\end{abstract}

\def\keywords2{\vspace{.5em}{\textbf{Mathematics Subject Classification
(MSC2000).}~\,\relax}}
\def\endkeywords2{\par}
\keywords2{{35B27, 35A08, 35B34, 35C20.}}

\def\keywords{\vspace{.5em}{\textbf{ Keywords.}~\,\relax}}
\def\endkeywords{\par}
\keywords{{Gradient meta-surface, subwavelength resonance, Helmholtz resonator, apparent full transmission and absorption, abrupt phase-shift.}}


\newcommand{\xS}{{x_{\textrm{S}}}}
\newcommand{\GSte}{{\Gamma_{\textrm{S}}}}
\newcommand{\lam}{\lambda}
\newcommand{\SxS}{\mathrm{S}_{\xS}^\lam}
\newcommand{\lgste}{\lambda^{\GSte}}
\newcommand{\GDir}{{\Gamma_{\textrm{D}}}}
\newcommand{\GNeu}{{\Gamma_{\textrm{N}}}}
\newcommand{\GDel}{{\Gamma_{\Delta}}}

\section{Introduction}\label{Ch:Introduction}


Surfaces covered by a microscopic structure can display unforeseen physical properties, which can be applied in the everyday life to have advantageous effects like anti-reflection coating and high-efficiency light absorbers, or enhancers. These are called meta-surfaces and have been subject to intensive research \cite{Lin298, Xingjie13, Lingling13, Nanfang14, HRSuperLens, maurel1, maurel2}.
We are particularly inspired by the findings in \cite{add, BroadbandAnomalousReflection}. In \cite{BroadbandAnomalousReflection}, the authors built a gradient meta-surface covered with microscopic small gold plates of different sizes, which are energized by an electric current, and afterwards the authors physically illuminated that surface by an incident light with a particular frequency and particular angle with respect to the ground, and emphasized a reflection with an altered outgoing angle. For certain angles of incidence, they achieved an absorption of the impinging wave. More importantly, they observed at the resonant frequencies an abrupt shift  of the phase  of the scattered wave. This unusual behaviour of the phase shift at muti-resonances is intrinsic to gradient meta-surfaces \cite{reviewmeta}. 

Here, we mainly consider acoustic waves. Our objective is to uncover the behaviour of a meta-surface built upon microscopic Helmholtz resonators, one such Helmholtz resonator can be any cavity with a small hole, as for example a water bottle with a small opening compared to its size. Such a water bottle admits a high pitched noise when blown upon, that is an effect of the resonance. To be more precise in our design, the structure is a periodically repeated ensemble of rectangular Helmholtz resonators each with a different size, where the opening is placed on the center of their ceiling. As it was the case in the paper with the gold plates, we let a plane wave hit the gradient meta-surface with a certain angle and frequency. We can expect that for  appropriate frequencies and appropriate angles of incidence, the resonances of the cavities interact with each other and produce a modified reflected wave. Although we work here in two dimensions only, it is enough as pointed out in \cite{BroadbandAnomalousReflection}, because there the third dimension, which corresponds to the depth, has no measurable influence.

A solution for the scattering problem is provided in Theorem \ref{thm:mainresults}. The numerical results show an abrupt phase-shift for the scattered wave. Applying the techniques developed in \cite{FWMSP1, HabibHai, BubbleMetaScreen}, we provide an asymptotic formula in which the resonance is involved in the scattered wave. That equation depends on the controllable parameters, i.e., the incident frequency and the lengths of the openings of the Helmholtz resonators. It also depends on a matrix, which captures the coupling between the Helmholtz resonators. Controlling such parameters, we can exploit the resonance effect, which leads to a precise description of the unusual scattering and absorption properties of the gradient meta-surface. This is further investigated in our numerical applications at the end of this paper. From the proof of Theorem \ref{thm:mainresults}, we can extract a formula for a numerical evaluation of the scattered wave in the near-field and there we observe other remarkable outcomes like an enhancement of the amplitude of the scattered field, which is reminiscent of the results in \cite{HRSuperLens}. Finally, it is worth emphasizing the connection between our results and those obtained for periodic arrays of narrow slits in the series of papers \cite{hai1, hai2, hai3, hai4, hai5, hai6}.

This paper is organized as follows. In Section \ref{sec:Prelim} we prepare the mathematical foundation for the main result. There we define the geometry of the structure, the equation for the plane wave and the partial differential equation which models the solution to the problem, and make a short remark about the physical meaning for the boundary conditions for acoustic walls. In Section \ref{Sec:MainResult}, we  show the main results and comment on the appropriate choices of parameters, like the coupling matrix  and the scaling factor and describe a way to utilize the resonance effect. In Section \ref{sec:MainProof}, we state and prove Theorem \ref{thm:mainresults}.  In Section \ref{sec:NumImplTest}, we first explain how we implemented the numerics in Matlab and how we tested them. Then we present numerical applications of the results of Theorem \ref{thm:mainresults} for some certain geometries. In Section \ref{sec:ConcludingRemarks}, we conclude the paper with final considerations, open questions and possible future research directions.

\makeatletter
\def\moverlay{\mathpalette\mov@rlay}
\def\mov@rlay#1#2{\leavevmode\vtop{%
   \baselineskip\z@skip \lineskiplimit-\maxdimen
   \ialign{\hfil$\m@th#1##$\hfil\cr#2\crcr}}}
\newcommand{\charfusion}[3][\mathord]{
    #1{\ifx#1\mathop\vphantom{#2}\fi
        \mathpalette\mov@rlay{#2\cr#3}
      }
    \ifx#1\mathop\expandafter\displaylimits\fi}
\makeatother

\newcommand{\cupdot}{\charfusion[\mathbin]{\cup}{\cdot}}
\newcommand{\bigcupdot}{\charfusion[\mathop]{\bigcup}{\cdot}}
\newcommand{\OmBar}{\overline{\Om}}
\newcommand{\ZxS}{\mathrm{Z}_{\xS}^k}
\newcommand{\ZxSD}{\mathrm{Z}_{\mathrm{D}, \xS}^k}
\newcommand{\ZxSN}{\mathrm{Z}_{\mathrm{N}, \xS}^k}
\newcommand{\lgdir}{\lambda^{\GDir}}
\newcommand{\lgfull}{\lambda^{\del\Om}}
\newcommand{\lgempt}{\lambda^{\varnothing}}
\newcommand{\Hank}[1]{\mathrm{H}^{(1)}_{#1}}
\newcommand{\Sobo}{{H}}
\newcommand{\Sobominhalf}{{H}^{-\nicefrac{1}{2}}}
\newcommand{\Soboplushalf}{{H}^{\nicefrac{1}{2}}}
\newcommand{\Hminhalftilde}{\widetilde{H}^{-\nicefrac{1}{2}}}
\newcommand{\HminhalftildeNull}{\widetilde{H}^{-\nicefrac{1}{2}}_{\langle 0\rangle}}
\newcommand{\Hplushalfast}{{H}^{\nicefrac{1}{2}}_{\ast}}
\newcommand{\SGDk}{\mathcal{S}_\GDir^k}
\newcommand{\SGNk}{\mathcal{S}_\GNeu^k}
\newcommand{\KGNkstar}{(\mathcal{K}^{k}_\GNeu)^{\ast}}
\newcommand{\KGGkstar}{(\mathcal{K}^{k}_{\del\Om})^{\ast}}
\newcommand{\dSGDk}{\del\mathcal{S}^{k}_{\GDir}}
\newcommand{\Gk}{\Gamma^k}
\newcommand{\calAk}{\mathcal{A}^k}
\newcommand{\calAe}{\mathcal{A}_\eps}
\newcommand{\calBe}{\mathcal{B}_\eps}
\newcommand{\calAz}{\mathcal{A}_0}
\newcommand{\psiDN}{\begin{bmatrix} \psi\MID_{\GDir}\\ \psi\MID_{\GNeu}\end{bmatrix}}
\newcommand{\calAkz}{\mathcal{A}^k_0}
\newcommand{\calAke}{\mathcal{A}^k_\eps}
\newcommand{\kjzer}{k_j^0}
\newcommand{\kjeps}{k_j^\eps}

\newcommand*\euler{\mathrm{e}}
\newcommand*\imagi{\mathrm{i}\,}
\newcommand*{\bk}{\mathbf{k}}
\newcommand{\colvec}[2][.8]{%
  \scalebox{#1}{%
    \renewcommand{\arraystretch}{.8}%
    $\begin{pmatrix}#2\end{pmatrix}$%
  }
}
\newcommand{\smallvec}[1]{\ensuremath
    \left(\begin{smallmatrix}#1\end{smallmatrix}\right)%
}
\newcommand{\compactvec}[1]{\ensuremath
    \big(\begin{smallmatrix}#1\end{smallmatrix}\big)%
}
\section{Preliminaries}\label{sec:Prelim}
We first fix the geometry of our problem, i.e., the lengths and heights of our Helmholtz resonators, and the positions and lengths of their openings. This allows us to provide a mathematical model for the resulting wave when we imping the structure with an incident wave with a particular angle to the ground and a particular wave vector. That model is built upon a partial differential equation on  a unit strip with an outgoing wave condition and a quasi-periodicity condition on both sides of the unit strips. With that in hand we then can state the main result of this paper.

\subsection{Geometry of the Problem}
Let $\delta>0$, denote a small parameter, let $N\in \NN$ denote the number of Helmholtz resonators, then we denote a single Helmholtz resonator in the unit strip $Y\DEF (-\nicefrac{\delta}{2},\nicefrac{\delta}{2}) \times (0,\infty)\subset \RR^2$ by $D_i \subset Y$, for $i\in\{ 1,\ldots,N\}$. We set every Helmholtz resonator $D_i$ to be a rectangle of height $h_i>0$ and of length $l_i>0$ and its center to be located at $(\xi_i, \nicefrac{h_i}{2})$, where $\xi_i \in (-\nicefrac{\delta}{2}+\nicefrac{l_i}{2},\;\nicefrac{\delta}{2}-\nicefrac{l_i}{2})$. Thus the lower boundaries of the Helmholtz resonators intersect the horizontal axis, but the boundaries on the side do not intersect the side of the unit strip. Every Helmholtz resonator $D_i$ has also an opening gap $\Lambda_i\DEF (\xi-\eps_i, \xi+\eps_i)\times \{h_i\}$  located at the center of its ceiling, that is at $(\xi_i, h_i)$ and it has a radius of $\eps_i>0$, hence a diameter of $2\,\eps_i$. This forms the geometry in our unit strip $Y$. See Figure \ref{fig:UnitStrip} for an illustration of an example with $N=3$. 

\begin{figure}[h]
  \centering
  \includegraphics[width=0.85\textwidth]{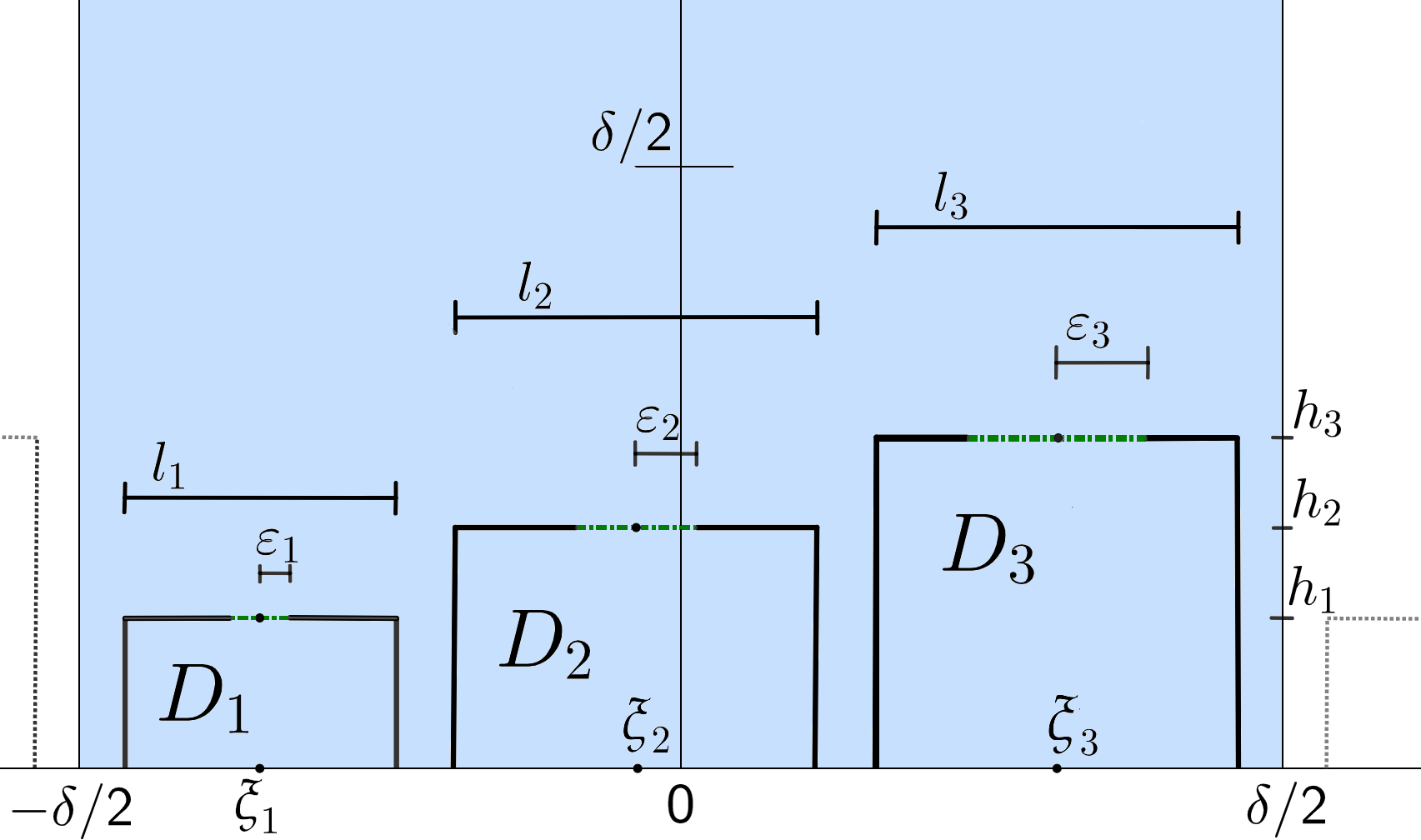}
  \caption{Depicted is the unit strip $Y$ in the color blue. For this example we have chosen $N=3$, thus we have the three Helmholtz resonators $D_1, D_2$ and $D_3$. The three in green dotted lines represent the three openings $\Lambda_1$, $\Lambda_2$ and $\Lambda_3$. }\label{fig:UnitStrip}
\end{figure}

Our gradient meta-surface in $\RR^2$ is built upon periodically repeating the unit strip along the horizontal axis. We are now interested in the result when we imping our gradient meta-surface with a harmonic incident plane wave $U_0^k:\RR^2\rightarrow\CC$ given through
\begin{align*}
	U_0^k(x)\DEF I_0\,\euler^{-\imagi k_1\,x_1}\euler^{-\imagi k_2\,x_2}\,,
\end{align*}
where $\imagi$ is the imaginary unit, $I_0\in\RR$ denotes the intensity of the incident wave, and $k_1\in \RR$ and $k_2<0$ are the horizontal, respectively vertical, components of the wave vector $(k_1, k_2)\in\RR^2$. Let $k^2\DEF k_1^2+k_2^2$ be the length of the wave vector. See Figure \ref{fig:GradSurface} for an illustration of the gradient meta-surface. 

\begin{figure}[h]
  \centering
  \includegraphics[width=0.8\textwidth]{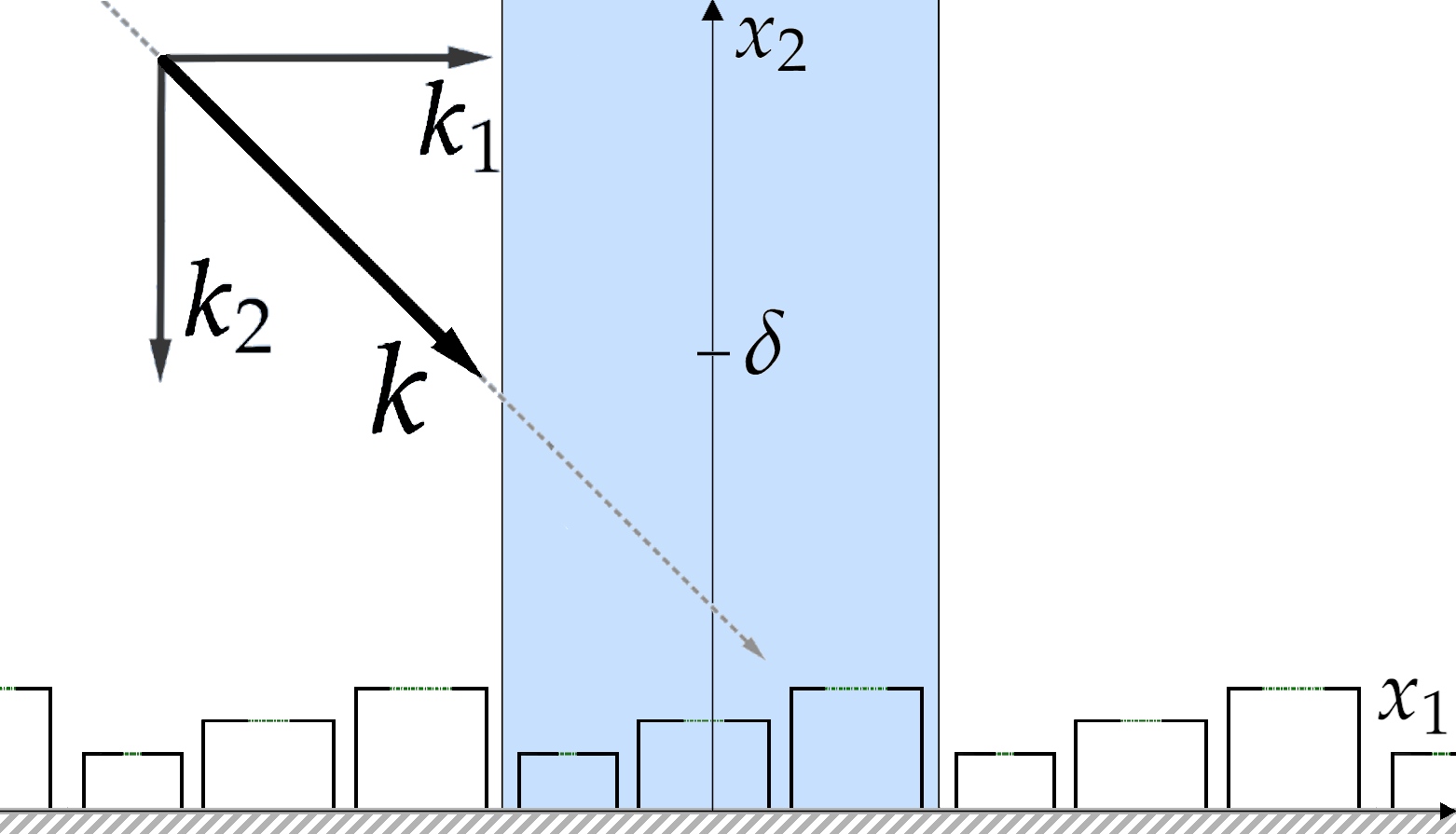}
  \caption{This figure show the unit strip repeated along the $x_1$ axis. The wave vector has the two components $k_1\in\RR$ and $k_2<0$ and the length $k$.}\label{fig:GradSurface}
\end{figure}

\subsection{Mathematical Model for the Scattering Problem}
We denote the wave function, which results from letting our geometry interacts with the incident wave, by $U^k: Y \rightarrow\CC$. We model $U^k$ as a solution to the following partial differential equation:
\begin{align}\label{PDE:Uk}
	\left\{ 
	\begin{aligned}
		 \left( \Laplace + k^2  \right) \Uk &= 0 \quad &&\text{in } \RR_+^2\DEF\{x\in\RR^2\,|\,x_2\geq 0\}\,, \\
		 \Uk \MID_+\! - \Uk\MID_- &= 0 \quad &&\text{on all } \Lambda_i\,,\\
		 \del_\nu \Uk \MID_+\! - \del_\nu \Uk\MID_- &= 0 \quad &&\text{on all }  \Lambda_i\,,\\
		 \del_\nu \Uk \MID_+\! &= 0 \quad &&\text{on all }  \del D_i \setminus \Lambda_i \,,\\
		 \del_\nu \Uk \MID_-\! &= 0 \quad &&\text{on all } \del D_i \setminus \Lambda_i \,,\\
		 \Uk&=0 \quad &&\text{on }\del \RR^2_+ \setminus \cup_{i=1}^N\del D_i \,,
	\end{aligned}
	\right.
\end{align}
where $\cdot\!\mid_+$ denotes the limit from outside of $ D_i$ and $\cdot\!\mid_-$ denotes the limit from inside of $D_i$, and $\del_\nu$ denotes the outside normal derivative on $\del D_i$, for all $i\in\{1,\ldots,N\}$. The first condition, that is, the Helmholtz equation, represents the time-independent wave equation, and arises from the wave equation by separation of variables in time and spatial domain. The second and third conditions represent the transmission conditions. The forth and fifth conditions represent rigid walls. It is also known as the Neumann condition. For acoustic waves, the sixth condition represents a sound cancelling ground layer. It is also known as the Dirichlet condition. 

Similar to diffraction problems for gratings (see, for instance, \cite{MCMP, Bao95} and the references therein), the above system of equations is completed by a certain outgoing radiation condition imposed on the scattered field $\Uks\DEF \Uk-\Uknull$ and a quasi-periodicity condition on $\Uk$, that is,
\begin{align*}
	\big| \del_{x_2} \Uks -\imagi k_2\Uks \big| 
		&\rightarrow 0\quad&&\text{for}\quad x_2\rightarrow\infty\,,\\
   \Uk\big( x+ {\smallvec{\delta\\0}} \big) 
   		&= \euler^{-\imagi k_1 \delta}\,\Uk(x) \quad&&\text{for}\quad x\in \RR_+^2\,.
\end{align*}
Both conditions follow from asserting that the scattered wave $U^k_s$ behaves in the same way in the far-field and in periodicity as the horizontally reflected incident wave. We remark here that the horizontally reflected incident wave, that is the function $I_0\euler^{-\imagi k_1 x_1}\euler^{\imagi k_2 x_2}$, is the scattered solution to (\ref{PDE:Uk}) in absence of Helmholtz resonators.

We further remark that in the general case, where the incident wave is a superposition of plane waves, we can decompose the incident wave using Bloch-Floquet theory, see, for instance, \cite{FlocquetTheory, LPTSA}. We obtain a family of problems to solve, each one with its own outgoing radiation condition. The final solution is then the superposition of all of these solutions.

\newcommand{\rijhdel}{{\mathrm{r}_{i,j}^{h,\del}}}
\newcommand{\rijldel}{{\mathrm{r}_{i,j}^{l,\del}}}
\newcommand{\ridel}{{\mathrm{r}_{i}^{\del}}}
\newcommand{\ri}{{\mathrm{r}_{i}}}
\newcommand{\rex}{{\mathrm{r}^{\text{ex}}}}
\newcommand{\rih}{{\mathrm{r}_{i}^h}}
\newcommand{\ril}{{\mathrm{r}_{i}^l}}
\newcommand{\rjh}{{\mathrm{r}_{j}^h}}
\newcommand{\rjl}{{\mathrm{r}_{j}^l}}
\newcommand{\rijh}{{\mathrm{r}_{i,j}^h}}
\newcommand{\rijl}{{\mathrm{r}_{i,j}^l}}
\newcommand{\fdk}{f^{\delta k}}
\def\lowcomma{_{\textstyle,}}

\section{Main Result}\label{Sec:MainResult}
We assume that 
\begin{align*}
	\delta k < \min \big( \tfrac{\pi}{l_i}\,, \tfrac{\pi}{h_i}\big)\,,
		\quad\text{for all }i\in\{1,\ldots, N\}\,,
\end{align*}
which originates from the fact that $\delta^2 k^2$ has to be smaller than the first non-zero Laplace eigenvalue with Neumann conditions for every Helmholtz resonator $D_i$.

Furthermore, we assume that
\begin{align*}
	\delta k<2\pi-|\delta k_1|\,.
\end{align*}
This is due to Lemma \ref{lemma:rdel,r,rex}, which expresses the behaviour of certain auxiliary functions. This condition can be relaxed, but then the formulas get more complicated.

\begin{theorem}\label{thm:mainresults}
	Let the matrix $R^{\delta k}\in\CC^{N\times N}$ be defined by (\ref{equ:R-decomp}), let the constants  $\ri\,,\ridel \in \CC$, for $i\in\{1,\ldots N\}$, and the constant $\rex \in \CC$ be defined through Lemma \ref{lemma:rdel,r,rex}. Let $\eps\DEF\max(\eps_1,\ldots,\eps_N)>0$ be small enough. Then for all $z\in Y$, where $z_2$ is large enough, we have that
	\begin{align*}
		\Uk(z)-&\Uk_0(z) = I_s\;
			\euler^{-\imagi k_1 z_1}\, 
			\euler^{\imagi k_2 z_2}
			+ \OO(N\,\eps\,\|Q(\delta k)^{-1}\|_2))
			+ \OO(\euler^{-C\,z_2})\,,
	\end{align*}
	for a constant $C>0$ independent of $z$, where $\| \cdot \|_2$ denotes the Frobenius norm for matrices and where $I_s\in\CC$ is given through
	\begin{align*}
		I_s = \sum_{i=1}^N 
					(Q(\delta k)^{-1}\,\fdk)_i\,\euler^{\imagi\delta k_1 \xi_i}
					\big(\ri-\frac{\sin(\delta k_2\,h_i)}{\delta k_2}\big)
			- I_0(1-2\imagi\,\delta k_2\,\rex)\,,
	\end{align*}
	where $\fdk\in\CC^N$ is a vector with the components
	\begin{align*}
		\fdk_i\DEF 2\imagi\,\delta k_2\,I_0\,\euler^{-\imagi \delta k_1\,\xi_i}
			\,\Big(
				\frac{2\,\sin(\delta k_2\,h_i)}{\delta k_2}-\ridel
			\Big)\,,
	\end{align*}
	$Q(\delta k)\in\CC^{N\times N}$ is a matrix given through
	\begin{align*}
		Q(\delta k)\DEF \tfrac{1}{(\delta k)^2}|D|^{-1}+\tfrac{2}{\pi}\,\mathrm{diag}\big(\log(\eps_i/2)\big) + \,R^{\delta k}\,,
	\end{align*}
	$\mathrm{diag}\big(\!\log(\eps_i/2)\big)$ is a $N\times N$ diagonal matrix with diagonal entries $\log(\eps_i/2)$, and where $|D|$ is a $N\times N$ diagonal matrix with diagonal entries $|D_i|$.
\end{theorem}

The numerical implementation of the matrix $R^{\delta k}$ and the constants $\ri\,,\ridel\,, \rex$ is explained in Section \ref{sec:NumImplementation}.

The value $\delta>0$ only appears in the form $\delta k_1$ or $\delta k_2$ in the theorem, thus we can use $\delta$ only to scale the incoming wave vector $(k_1,k_2)$. Consider that the geometry scales with $\delta$, and it does so in both spatial dimensions simultaneously.

The value $I_s$ depends on the inversion of an $N\times N$ matrix $Q(\delta k)$, and for some complex values of $\delta k$, which are close to the physical resonance values of the system, we obtain a blow-up in the entries of $Q(\delta k)^{-1}$. In our set-up we do not allow for a complex wave vector, only for real values for $\delta k$, thus we never encounter a blow-up of $Q(\delta k)^{-1}$. Nonetheless, when $\delta k$ approaches the real part of those singular complex values, we see numerically a local extremum in the entries of $Q(\delta k)^{-1}$ and hence a significant effect of the gradient meta-surface on the scattered wave. 

Now $Q(\delta k)$ is of the form: a diagonal matrix added to $R^{\delta k}$, and we can vary 
the entries in the diagonal matrix by varying all $\eps_i$. This gives us a control on the singular behaviour of $Q(\delta k)$. Intuitively, to achieve a highly singular behaviour, we require all the eigenvalues of $Q(\delta k)$ to be close to zero, which we obtain when the entries of the diagonal matrix $\frac{1}{(\delta k)^2}|D|^{-1}+\frac{2}{\pi}\,\mathrm{diag}(\log(\eps_i/2))$ are close to the opposite of the eigenvalues of $R^{\delta k}$. In the simplest case, $R^{\delta k}$ has one real eigenvalue $\lambda_{R^{\delta k}}$ of multiplicity $N$, then we vary those $\eps_i$ so that $-\big(\frac{1}{|D_i|(\delta k)^2}+\frac{2}{\pi}\,\log(\eps_i/2)\big)$ is equal to that eigenvalue for all $i=1,\ldots,N$. This leads to the equation
\begin{align*}
	\eps_i = 2\,\exp\Big(\!-\tfrac{\pi}{2}\big(\tfrac{1}{|D_i|(\delta k)^2}+\lambda_{R^{\delta k}}\big)\Big)\,.
\end{align*}
Since $R^{\delta k}$ will generally not have only one real eigenvalue with multiplicity $N$, we also vary $(\delta k)^2$ in the hope to obtain an eigenvalue with multiplicity at least $2$. In the numerical simulations we elaborate on how we determine those $\eps_i$ on the considered geometries. 

Note that the denser the eigenvalues of $R^{\delta k}$ are, the broader is the range of frequencies for which one can observe a significant effect of the gradient meta-surface on the scattered wave.

\newcommand{\Gkm}{{\Gamma_-^k}}
\newcommand{\Gop}{{\Gamma_+^0}}
\newcommand{\RkGm}{{\mathrm{R}_{\Gamma , +}^k}}
\newcommand{\Nki}{{\mathrm{N}_{i}^k}}
\newcommand{\Rki}{{\mathrm{R}_{i}^k}}
\newcommand{\Nkp}{{\mathrm{N}_{+}^k}}
\newcommand{\Nkm}{{\mathrm{N}_{-}^k}}
\newcommand{\Rkp}{{\mathrm{R}_{+}^k}}
\newcommand{\Rkm}{{\mathrm{R}_{-}^k}}

\newcommand{\Nkidel}{{\mathrm{N}_{i,\,\del}^k}}
\newcommand{\Rkidel}{{\mathrm{R}_{i,\,\del}^k}}
\newcommand{\Nkpdel}{{\mathrm{N}_{+,\,\del}^k}}
\newcommand{\Nkmdel}{{\mathrm{N}_{-,\,\del}^k}}
\newcommand{\Rkpdel}{{\mathrm{R}_{+,\,\del}^k}}
\newcommand{\Rkmdel}{{\mathrm{R}_{-,\,\del}^k}}

\newcommand{\Gdkm}{{\Gamma_-^{\delta k}}}
\newcommand{\RdkGp}{{\mathrm{R}_{\Gamma , +}^{\delta k}}}
\newcommand{\Ndki}{{\mathrm{N}_{i}^{\delta k}}}
\newcommand{\Rdki}{{\mathrm{R}_{i}^{\delta k}}}
\newcommand{\Ndkp}{{\mathrm{N}_{+}^{\delta k}}}
\newcommand{\Ndkm}{{\mathrm{N}_{-}^{\delta k}}}
\newcommand{\Rdkp}{{\mathrm{R}_{+}^{\delta k}}}
\newcommand{\Rdkm}{{\mathrm{R}_{-}^{\delta k}}}

\newcommand{\Ndkidel}{{\mathrm{N}_{i,\,\del}^{\delta k}}}
\newcommand{\Rdkidel}{{\mathrm{R}_{i,\,\del}^{\delta k}}}
\newcommand{\Ndkpdel}{{\mathrm{N}_{+,\,\del}^{\delta k}}}
\newcommand{\Ndkmdel}{{\mathrm{N}_{-,\,\del}^{\delta k}}}
\newcommand{\Rdkpdel}{{\mathrm{R}_{+,\,\del}^{\delta k}}}
\newcommand{\Rdkmdel}{{\mathrm{R}_{-,\,\del}^{\delta k}}}

\newcommand{\curlRdk}{\mathcal{R}^{\delta k}}

\newcommand{\LLetilde}{\tilde{\LLe}}
\newcommand{\Ltilde}{\widetilde{{L}}}

\section{Proof of the Main Results}\label{sec:MainProof}
The outline of the proof is as follows. Before we can begin with the actual idea to prove the result, which starts in Subsection \ref{sec:microscopicview}, we need a bundle of auxiliary functions, like the Neumann functions, which describe the solution to (\ref{PDE:Uk}) with closed openings, and for these Neumann functions we need the fundamental solution to the Helmholtz equation, which we denote here by the Gamma-function. In Subsection \ref{sec:IntegralEqu}, we formulate integral equations, which allow for a numeric evaluation of these Neumann functions.

Then we begin the proof by defining a scaled form of the resulting wave vector. With that microscopic view, we use Green's identity and the conditions from (\ref{PDE:Uk}) to establish an integral equation on functions defined on the openings, which is solvable due to an invertible hypersingular integral operator, defined in Subsection \ref{sec:LLe}. Its solution depends on the constant matrix $R^k$ given in Theorem \ref{thm:mainresults}. Using Green's identity again, we can recover the behaviour of the resulting wave on the far-field from its behaviour on the openings. Re-establishing the macroscopic view, we have then proven Theorem \ref{thm:mainresults}.

\subsection{The Gamma-Function and the Neumann-Function}\label{sec:GammaFctNeumannFct}
Let the wave vector satisfy the condition
\begin{align} \label{assump:k^2<inf}
	k^2 < \inf_{n\in\ZZ\setminus\{0\}}\{| 2\pi\,n-k_1 |^2\}\,.
\end{align} 
Then we define the quasi-periodic fundamental solution $\Gkp$ to the Helmholtz operator $\Laplace +k^2$ by
\begin{align}\label{equ:GammaPlus}
	\Gkp(z,x)
		&\DEF\frac{\euler^{-\imagi k_1  (z_1-x_1)} \big(\euler^{\imagi k_2  \, |z_2-x_2|}-\euler^{\imagi k_2  \, |z_2+x_2|}\big)}{2 \imagi k_2} \nonumber\\
		& - \mkern-15mu \sum_{\substack{n\in\ZZ\setminus\{0\}\\ s_n\DEF \sqrt{| 2\pi\, n - k_1|^2 - k^2}}} 
		  \mkern-25mu\frac{\euler^{-\imagi (2\pi\, n-k_1) (z_1-x_1)}\big(\euler^{- s_n | z_2-x_2 |}-\euler^{- s_n | z_2+x_2 |}\big)}{2\,s_n}
\end{align}
for $z, x \in Y$, $z\neq x$.
The sum in Equation (\ref{equ:GammaPlus}) can also be defined if $k$ does not satisfy the above condition (\ref{assump:k^2<inf}), for this we refer to \cite[Section 3.1]{BubbleMetaScreen}. $\Gkp$ satisfies the equation $(\Laplace_x +k^2)\Gkp(z,\,x)=\delta_z(x)$ for $x\in Y$, where $\delta_z(\cdot)$ denotes the Dirac mass in $z$, and $\Gkp$ is quasi-periodic in its second variable, that is, $\Gkp(z, x+ {\smallvec{\delta\\0}})=\euler^{-\imagi k_1\delta}\,\Gkp(z, x)$. The $+$ in $\Gkp$ originates from the fact that for $x_2$ large enough, $\Gkp(z,x)$ is proportional to $\euler^{+ \imagi k_1 x_1}$ for a fixed $x_2$. 
For $x$ close enough to $z$ we have that 
\begin{align}
	\Gkp(z,x) = \frac{1}{2\pi}\log(|x-z|) + \RkGm(z,x)\,,
\end{align}
where $\RkGm(z,\cdot)\in H^{\nicefrac{3}{2}}_\text{loc}(\RR^2)$, this follows from describing $\Gkp$ as an infinite sum of Hankel functions of the first kind of order 0, see \cite[Section 3.1]{BubbleMetaScreen}, and the singularity term of that Hankel function, see \cite[Section 2.3]{LPTSA}. Moreover, we have that for $k\rightarrow 0$, $\Gkp\rightarrow\Gop$, where $\Gop$ is the fundamental solution to the Laplace equation, compare \cite[Section 3.1]{BubbleMetaScreen}.

Next we define the Neumann-functions $\Nkp(z,x)$ and $\Nki(z,x)$ for $i\in\{1,\ldots,N\}$. $\Nki(z,x)$ satisfies the equation $\Laplace_x+k^2=\delta_z$ inside $D_i$, with the boundary condition $\del_\nu \Nki(z,\cdot)=0$. Since $D_i$ is assumed to be a rectangle, we can express $\Nki(z,x)$ as a conditionally convergent sum for $x\in D_i$. This can be readily assembled through \cite[Proposition 2.7]{LPTSA} and \cite[Section 3.1]{LaplEigenValuesFunctions}. We also see that $\Nki$ is well-defined except when  $k^2$ is a Laplace eigenvalue with Neumann boundary conditions. We exclude these values for $k$ from now on. We can express $\Nki(z,x)$ for $x$ close enough to $z$, and $k$ smaller than the first non-zero Laplace eigenvalue with the Neumann boundary condition through
\begin{align}\label{equ:Nki_splitt}
	\Nki(z,x) = \frac{1}{2\pi}\log(|x-z|)+ \frac{1}{k^2}\frac{1}{|D_i|} + \Rki(z,x)\,,
\end{align}
where $\Rki(z,\cdot)\in H^{\nicefrac{3}{2}}(D_i)$, see \cite[Lemma 2.9]{LPTSA}, with $H^{\nicefrac{3}{2}}$ being a Sobolev space. We know that $\Rki(z,\cdot)$ is analytic in $k$ for $k$ smaller than the first non-zero Laplace eigenvalue with Neumann conditions in the domain $D_i$, see \cite[Proposition 2.7]{LPTSA}.

The function $\Nkp(z,x)$, $z,x\in Y\setminus \overline{\cup_{i=1}^n D_i}$ is defined as the solution to 
\begin{align}\label{PDE:Nkp}
	\left\{ 
	\begin{aligned}
		 \left( \Laplace + k^2  \right) \Nkp(z,\cdot) &= \delta_z(\cdot) \quad &&\text{in } Y\setminus \overline{\cup_{i=1}^n D_i}, \\
		 \del_\nu \Nkp(z,\cdot) \MID_+\! &= 0 \quad &&\text{on } \del D_i,\\
		 \Nkp(z,\cdot)&=0 \quad &&\text{on } \del \RR^2_+ ,
	\end{aligned}
	\right.
\end{align}
where it is additionally required that the quasi-periodicity condition is satisfied, that is, $\Nkp(z, x+ {\smallvec{\delta\\0}})=\euler^{-\imagi k_1\delta}\,\Nkp(z, x)$, as well as the radiation condition, that is, $\big| \del_{x_2} \Nkp(z,x) -\imagi k_2\Nkp(z,x) \big| \rightarrow 0$ for $x_2\rightarrow\infty$. We note here that the normal to $\del D_i$ still points outside. We can write the Neumann-function $\Nkp$ as 
\begin{align*}
	\Nkp(z,x) = \Gkp(z,x)+\Rkp(z,x)\,,
\end{align*}
where the remainder $\Rkp$ exists in $H^1_{\text{loc}}(Y)$, we refer to \cite[Chapter 7]{LPTSA}. An integral equation is given in the next section.

Until now we have assumed that the variable $z$ is an element in an open domain, but we also need that $z$ is located at the boundary. When $z$ is at the boundary, the Dirac measure in the Helmholtz equation has its mass on the boundary as well and since an integral over the domain covers only half that mass, we have to multiply by a factor of 2 to recover the whole mass, and with that, the spatial singularity at $z$ is twice the singularity when $z$ is in the domain. This changes the Neumann function, as well as the remainders. We denote these variations by
$\Nkidel, \Rkidel, \Nkpdel, \Rkpdel$.
For example, it holds now that for $z\in\cup_{i=1}^N\del D_i$
\begin{align*}
	\Nkpdel(z,x) 
		&= 2\Gkp(z,x)+\Rkpdel(z,x)\, \quad\text{for }x\in Y\setminus \overline{\cup_{i=1}^n D_i}\,,\\
	\Nkidel(z,x) 
		&= \frac{1}{\pi}\log(|x-z|)+ \frac{1}{k^2}\frac{1}{|D_i|} + \Rkidel(z,x)\,.
\end{align*}
Note here that the singularity in $k$ has no factor of $2$. We see this by considering that the function $\Nkidel(z,\,\cdot)-\frac{1}{\pi}\log(|z-\cdot|)$ satisfies the Helmholtz equation with a zero right-hand side with smooth Neumann boundary conditions and thus can be formulated as the convolution of $\Nki$ and the Neumann boundary data, see \cite[Section 2.3.5.]{LPTSA}. Since $\Nki$ has the singularity $\frac{1}{k^2}\frac{1}{|D_i|}$ so does then $\Nkidel$.

This shows that the quotient between $\Rki$ and $\Rkidel$ as well as the quotient between $\Rkp$ and $\Rkpdel$ might not have a simple relation like them being equal to $2$ or $\nicefrac{1}{2}$ after taking a limit to the boundary.

\subsection{Integral Equation for $\Rkp$ and $\Rkpdel$} \label{sec:IntegralEqu}
Let us consider the integral equation for $\Rkpdel(z,x)$ first. Using Green's identity on $\int (\Laplace +k^2) \mathrm{N}^{(-k_1,k_2)}_{+,\del}(x,y)\,\Rkpdel(z,y)$ for $z,x\in\cup_i\del D_i$, over the unit strip $Y$, but outside the Helmholtz resonators, we obtain with the splitting $\mathrm{N}^{(-k_1,k_2)}_{+,\del}(x,y) = 2\,\Gamma^{(-k_1,k_2)}_{+}(x,y) + \mathrm{R}^{(-k_1,k_2)}_{+,\del}(x,y)$ the integral equation
\begin{multline*}
	\mathrm{R}^{(k_1,k_2)}_{+,\del}(z,x) 
		+2\int_{\cup_i\del D_i}\mathrm{R}^{(-k_1,k_2)}_{+,\del}(x,y)\del_{\nu_y}\Gkp(z,y)\intd\sigma_y\\
		= -4\int_{\cup_i\del D_i}\Gamma_{+}^{(-k_1,k_2)}(x,y)\del_{\nu_x}\Gkp(z,y)\intd\sigma_y\,,
\end{multline*}
where we used the quasi-periodicity condition, the Dirichlet ground condition and the radiation condition.
Here we note that the first-named integral has $\mathrm{R}^{(-k_1,k_2)}_{+}$ instead of $\mathrm{R}^{(k_1,k_2)}_{+}$. This is due to the quasi-periodicity conditions. We can correct this, by using Green's identity once again to obtain that 
\begin{multline*}
2\int_{\cup_i\del D_i}\!\!\!\!\mathrm{R}^{(-k_1,k_2)}_{+,\del}(x,y)\del_{\nu_y}\Gkp(z,y)\intd\sigma_y = 2\int_{\cup_i\del D_i}\!\!\!\!\mathrm{R}^{k}_{+,\del}(z,y)\del_{\nu_y}\Gamma^{(-k_1,k_2)}_{+}(x,y)\intd\sigma_y\,,
\end{multline*}
which then leads to the final integral equation for $\Rkpdel$ for $z,x\in\cup_i\del D_i$,
\begin{multline}\label{equ:integralEqu-Rkpdel}
	\mathrm{R}^{k}_{+,\del}(z,x) 
		+2\int_{\cup_i\del D_i}\mathrm{R}^{k}_{+,\del}(z,y)\del_{\nu_y}\Gamma^{(-k_1,k_2)}_{+}(x,y)\intd\sigma_y\\
		= -4\int_{\cup_i\del D_i}\Gamma_{+}^{(-k_1,k_2)}(x,y)\del_{\nu_x}\Gkp(z,y)\intd\sigma_y\,.
\end{multline}

With this at hand, we can formulate an equation for $x$ being outside of all $D_i$, through analogous arguments, and obtain
\begin{multline}\label{equ:integralEqu-Rkpdel:xoutside}
	\mathrm{R}^{k}_{+,\del}(z,x) 
		+\int_{\cup_i\del D_i}\mathrm{R}^{k}_{+,\del}(z,y)\del_{\nu_y}\Gamma^{(-k_1,k_2)}_{+}(x,y)\intd\sigma_y\\
		= -2\int_{\cup_i\del D_i}\Gamma_{+}^{(-k_1,k_2)}(x,y)\del_{\nu_x}\Gkp(z,y)\intd\sigma_y\,.
\end{multline}

If $z$ is outside of all $\overline{D_i}$, the factor 2 does not appear in the splitting of the Neumann function, and we can still use the same arguments as before and obtain for $\mathrm{R}^{k}_{+}$ for $z$ outside of all $\overline{D_i}$ and $x$ on the boundary that
\begin{multline}\label{equ:integralEqu-Rkp}
	\mathrm{R}^{k}_{+}(z,x) 
		+2\int_{\cup_i\del D_i}\mathrm{R}^{k}_{+}(z,y)\del_{\nu_y}\Gamma^{(-k_1,k_2)}_{+}(x,y)\intd\sigma_y\\
		= -2\int_{\cup_i\del D_i}\Gamma_{+}^{(-k_1,k_2)}(x,y)\del_{\nu_x}\Gkp(z,y)\intd\sigma_y\,,
\end{multline}
and if $x$ is outside the boundary we have 
\begin{multline}\label{equ:integralEqu-Rkp:xoutside}
	\mathrm{R}^{k}_{+}(z,x) 
		+\int_{\cup_i\del D_i}\mathrm{R}^{k}_{+}(z,y)\del_{\nu_y}\Gamma^{(-k_1,k_2)}_{+}(x,y)\intd\sigma_y\\
		= -\int_{\cup_i\del D_i}\Gamma_{+}^{(-k_1,k_2)}(x,y)\del_{\nu_x}\Gkp(z,y)\intd\sigma_y\,.
\end{multline}

With these integral equations (\ref{equ:integralEqu-Rkpdel}) to (\ref{equ:integralEqu-Rkp:xoutside}) and the formula for the propagating part in $\Gkp$, see Equation (\ref{equ:GammaPlus}), we can readily show the following lemma:
\begin{lemma}\label{lemma:rdel,r,rex}
	Let $k$ satisfy Condition (\ref{assump:k^2<inf}). There exists a $C>0$ such that for all $z\in \Lambda_i$ and for $x_2\rightarrow \infty$ we have that
	\begin{align} \label{equ:rdel}
		\Rkpdel(z,x)
			=\euler^{-\imagi k_1 z_1}\, \euler^{\imagi k_1 x_1}\, \euler^{\imagi k_2 x_2}
			\;\ridel+\OO(\euler^{-C\, x_2})\,,
	\end{align}	
	where $\ridel$ does not depend on $z$. And  
	for all $x\in \Lambda_i$ and $z_2\rightarrow \infty$ we have that
	\begin{align}\label{equ:ri}
		\Rkp(z,x) = \euler^{-\imagi k_1 z_1}\, \euler^{\imagi k_2 z_2}\, \euler^{\imagi k_1 x_1}\;\ri+\OO(\euler^{- C\,z_2})\,,
	\end{align}
	where $\ri$ does not depend on $z$. And
	for all $z$ outside of all $\overline{D_i}$ and $x_2\rightarrow \infty$ we have that
	\begin{align}\label{equ:rex}
		\Rkp(z,x) =\euler^{-\imagi k_1 z_1}\, \euler^{\imagi  k_1 x_1}\,\euler^{\imagi  k_2 z_2}\, \euler^{\imagi k_2 x_2}\,\rex +\OO(\euler^{- C\,(x_2-z_2)})\,,
	\end{align}
	where $\rex$ does not depend on $z$ or $x$.
\end{lemma}
Here we note that if $k$ does not satisfy  Condition (\ref{assump:k^2<inf}), then the propagating part in $\Gkp$ has a different form, which then will alter the formulas in the above lemma as well. We do not consider that case in this paper.

\subsection{The Hyper-Singular Operator $\LLe$} \label{sec:LLe}

Let $\mu'$ represent the distributional derivative of $\mu$. We define the following spaces and their respective norms:
\begin{align*}
   \curlXe
       &\DEF \left\{ \mu\in L^2((-\eps,\eps))\,\middle|\, \int_{-\eps}^\eps \sqrt{\eps^2-t^2}\,|\mu(t)|^2 \intd t \!<\! \infty \right\}\,,\\
   \curlYe
       &\DEF \left\{ \mu\in\cC^0([-\eps,\eps])\,\middle|\, \exists \,\mu'\in\curlXe \right\}\,,\\
   \NORM{\mu}_{\curlXe} 
       &\DEF \left( \int_{-\eps}^\eps \sqrt{\eps^2-t^2} \,|\mu(t)|^2\intd t \right)^{1/2},\,\\
   \NORM{\mu}_{\curlYe}
       &\DEF\left( \NORM{\mu}^2_{\curlXe} + \NORM{\mu'}^2_{\curlXe} \right)^{1/2}\,.
\end{align*}
Let $\eps>0$. Then we define the operator $\LLe :\curlXe \to \curlYe$ as
\begin{align}\label{def:OperatorsLLe}
   \LLe[\mu](\tau)\DEF&\int_{-\eps}^\eps \mu(t)\log(|\tau-t|)\intd t\,.
\end{align}
Let $\eps>0$ be small enough. Then the operator $\LLe :\curlXe \to \curlYe$ is bijective, see \cite[Chapter 11]{SaranenVainikko2002}, and we have that
\begin{align}\label{def:LLeinvOne}
   \LLeinv[1](t) = \frac{1}{\pi\,\log(\eps/2)\,\sqrt{\eps^2-t^2}}\,.
\end{align}
A general closed form for $\LLeinv$ is given in \cite[Proposition 3.11]{FWMSP1}, from which the function $\LLeinv[1](t)$ can be determined.

\subsection{The Microscopic View} \label{sec:microscopicview}
We define $\udk:\RR^2_+\to \CC$ by $\udk(x)\DEF U^k(\delta \,x)$, where $U^k$ is the resulting wave function, and analogously $\udknull:\RR^2_+\to \CC$, $\udknull(x)\DEF U^k_0(\delta \,x)$.
We readily see that $\udk$ satisfies (\ref{PDE:Uk}) with a scaled geometry and a scaled wave vector $(\delta k_1, \delta k_2)$. Furthermore, we have the following radiation condition and quasi-periodicity condition:
\begin{align*}
	\big| \del_{x_2} \udks -\imagi \delta k_2 \udks \big| 
		&\rightarrow 0\quad&&\text{for}\quad x_2\rightarrow\infty\,,\\
   \udk\big( x+ {\smallvec{1\\0}} \big) 
   		&= \euler^{-\imagi k_1 \delta}\,\udk(x) \quad&&\text{for}\quad x\in \RR_+^2\,,
\end{align*}
where $\udks(x)\DEF\udk(x)-\udknull(x)$. To reduce the amount of symbols we uphold the notation for $Y, D_i, \Lambda_i, l_i, h_i, \xi_i, \eps_i$ after scaling.

\subsection{Collapsing the Wave-Function to the Openings} \label{sec:Collaps}
\begin{lemma} \label{lemma:Gap-Formula}
	Let $i,j\in \{1,\ldots,N\}, i\neq j$ and let $z\in \Lambda_i$. Then we have that
	\begin{align*}
		\int_{\Lambda_i}\del_\nu \udk(y)
				\Big(\Ndkpdel(z,y)+\Ndkidel(z,y)\Big)
		\intd\sigma_y
		+
		\sum_{\substack{j=1\\j\neq i}}^{N}
		\int_{\Lambda_j}\del_\nu \udk(y)
				\Ndkpdel(z,y)
		\intd\sigma_y \\
		=
		2\imagi\,\delta k_2\,I_0\,\euler^{-\imagi \delta k_1\,z_1}
		\Big(
			2\,\frac{\sin(\delta k_2\,z_2)}{\delta k_2}-\ridel
		\Big)\,.
	\end{align*}
\end{lemma}
\begin{proof}
	Using Green's identity with the Neumann function $\Ndkidel$ inside $D_i$, we readily obtain for $z\in\Lambda_i$ that
	\begin{align} \label{equ:u inside D}
		\udk(z)=-\int_{\Lambda_i}\del_\nu \udk(y)\Ndkidel(z,y)\intd\sigma_y\,.
	\end{align}
	Analogously we apply Green's identity to $\udk$ using $\Ndkpdel$ on $Y$ with the properties mentioned in Section \ref{sec:GammaFctNeumannFct} and obtain
	\begin{align*}
		\udk(z)= \sum_{j=1}^N \int_{\Lambda_j}
			& \del_\nu \udk(y)\Ndkpdel(z,y)\intd\sigma_y\\
			& + 2\imagi \delta k_2 \lim_{r\rightarrow+\infty}
				\int_{\nicefrac{-1}{2}}^{\nicefrac{1}{2}}
					\udknull\big({\smallvec{y_1\\r}}\big)\Ndkpdel\big(z,{\smallvec{y_1\\r}}\big)
				\intd y_1\,.
	\end{align*}
	For an explicit demonstration we refer to \cite[Proof of Proposition 3.5]{FWMSP1}. Using Equation (\ref{equ:rdel}), that is, 
	$\Rdkpdel(z,x)=\euler^{-\imagi\delta k_1 z_1}\, \euler^{\imagi\delta k_1 x_1}\, \euler^{\imagi\delta k_2 x_2} \,\ridel+\OO(\euler^{-C x_2})$ 
	for $x_2\rightarrow\infty$ and $z_2=h_i$, where $\ridel$ does not depend on $x$, we can infer that 
	\begin{align*}
		\lim_{r\rightarrow+\infty}
				\int_{\nicefrac{-1}{2}}^{\nicefrac{1}{2}}
					\udknull\big({\smallvec{y_1\\r}}\big)\Ndkpdel\big(z,{\smallvec{y_1\\r}}\big)
				\intd y_1
		=I_0\,\euler^{-\imagi \delta k_1\,z_1}
		\big(
			 \ridel-2\,\frac{\sin(\delta k_2\,z_2)}{\delta k_2}
		\big)\,.
	\end{align*}
	Setting the first and second equations to be equal for $z\in\Lambda_i$ and applying the last equation to the second one, we prove the lemma.
\end{proof}

We define then $\mu_i\DEF \del_\nu \udk|_{\Lambda_i}$, and $\mu \DEF (\mu_1,\ldots,\mu_N)$ as well as $\curlRdk: (\curlXe)^N \to (\curlYe)^N $
\begin{align*}
	\curlRdk[\eta]
		\DEF\Big(
			\sum_{j=1}^N \curlRdk_{1, j}[\eta_j]\;,
			\;\ldots\;,
			\sum_{j=1}^N \curlRdk_{N, j}[\eta_j]
		\Big)\,,
\end{align*}
where
\begin{align*}
	\curlRdk_{i, j}[\eta_j](\tau)\DEF
		\left\{
			\begin{aligned}
		 		 \int_{-\eps_i}^{\eps_i} \Big[
		 		 		&	\RdkGp\big({\compactvec{z_i\\h_i}}, {\compactvec{y_i\\h_i}}\big)
		 		 			+ \Rdkpdel\big({\compactvec{z_i\\h_i}}, {\compactvec{y_i\\h_i}}\big) \\
		 		 		&\mkern80mu + \Rdkidel\big({\compactvec{z_i\\h_i}}, {\compactvec{y_i\\h_i}}\big)
		 		 	\Big] \eta_i(t)
		 		 \intd t \, \quad &\text{for } i=j\,,\\
		 		 \int_{-\eps_j}^{\eps_j} 
		 		 		&\Ndkpdel\big({\compactvec{z_i\\h_i}}, {\compactvec{y_j\\h_j}}\big)\eta_j(t)
		 		 \intd t \, \quad &\text{for } i\neq j\,,
			\end{aligned}
		\right.
\end{align*}
with $y_i= \xi_i+t$ and $z_i=\xi_i+\tau$ for $i\in\{1,\ldots, N\}$. Using the singular decomposition formulas for $\Gdkp, \Ndkidel$ and $\Ndkpdel$, which are written down in Section \ref{sec:GammaFctNeumannFct}, we can rewrite the equation in Lemma \ref{lemma:Gap-Formula} as
\begin{align}\label{equ:OpeningEquation,nonMat}
	\frac{2}{\pi}\LLei[\mu_i](\tau)
		& +\frac{1}{(\delta k)^2\,|D_i|} \int_{-\eps_i}^{\eps_i}\mu_i(t)\intd t
		+\sum_{j=1}^N \curlRdk_{i, j}[\mu_j](\tau)\nonumber\\
	& = 2\imagi\,\delta k_2\,I_0\,\euler^{-\imagi \delta k_1\,(\xi_i+\tau)}
		\Big(
			2\,\frac{\sin(\delta k_2\,h_i)}{\delta k_2}-\ridel
		\Big)\,
\end{align}
for $\tau\in(-\eps,\eps)$ and $i\in\{1,\ldots,N\}$, as long as $\delta k$ is smaller than the first non-zero Laplace eigenvalue with Neumann boundary conditions. We define $\fdk_i\in\CC$ to be
\begin{align*}
	\fdk_i\DEF 2\imagi\,\delta k_2\,I_0\,\euler^{-\imagi \delta k_1\,\xi_i}
		\,\Big(
			2\,\frac{\sin(\delta k_2\,h_i)}{\delta k_2}-\ridel
		\Big)\,,
\end{align*} 
which expresses the zero order term of the right-hand side with respect to $\epsilon$.
To simplify notation, we define $\fdk\DEF (\fdk_1\,,\,\ldots\,,\, \fdk_N)^\TransT$, where the superscript $\,{}^\TransT$ denotes the transpose, and we define $\LLe$ and $|D|$ to be diagonal matrices with diagonal entries $\LLei$ and $|D_i|$ respectively. Using Taylor series, we can decompose $\curlRdk$ as
\begin{align}\label{equ:R-decomp}
	\curlRdk[\mu] = R^{\delta k} \int_{-\eps}^\eps \mu + \int_{-\eps}^\eps\big(\OO(\eps)\big)_{i, j=1}^{N, N} \;\mu\,,
\end{align}
where $R^{\delta k}\in\CC^{N\times N}$ depends on $\delta k$ and $\int_{-\eps}^\eps \mu  \DEF (\int_{-\eps_1}^{\eps_1} \mu_1\,,\ldots\,,\int_{-\eps_N}^{\eps_N}\mu_N)^\TransT$.

\subsection{Solving for $\del_\nu \udk$ on the Openings}
We can rewrite Equation (\ref{equ:OpeningEquation,nonMat}) in the matrix-form as 
\begin{align}\label{equ:Opening_Id_LLetildeFULL}
\frac{2}{\pi}\LLe[\mu]+\frac{1}{(\delta k)^2}|D|^{-1}\int_{-\eps}^\eps \mu +\curlRdk[\mu]=\fdk\,,
\end{align}
 up to some error of order $\eps$, where still $\mu_i\DEF \del_\nu \udk|_{\Lambda_i}$. We can solve that equation in matrix form for $\del_\nu \udk|_{\Lambda_i}$.



To this end, we consider the series expansion $\curlRdk[\mu] = R^{\delta k} \int_{-\eps}^\eps \mu + \int_{-\eps}^\eps\big(\OO(\eps)\big)_{i, j=1}^{N, N} \;\mu$, where $R^{\delta k}\in\CC^{N\times N}$, see Equation (\ref{equ:R-decomp}) and we will neglect the error terms given in $\curlRdk[\mu]$.
Then we can reformulate the equation as
\begin{align}\label{equ:Opening_Id_LLetilde}
	\frac{2}{\pi}\mu+(\LLe)^{-1}\big[(\tfrac{1}{(\delta k)^2}|D|^{-1}+R^{\delta k})\int_{-\eps}^\eps \mu\big] =(\LLe)^{-1}[\fdk]\,,
\end{align}
where $|D|$ is a square diagonal matrix with diagonal-entries $|D_i|$. Consider that $\LLeinv[v](t)=\frac{v}{\pi\,\log(\eps/2)\,\sqrt{\eps^2-t^2}}$ for a constant function with value $v\in\CC^N$, where $\eps=(\eps_1,\ldots,\eps_N)$ and the elementary operations are defined element-wise.

After applying $\int_{-\eps}^\eps$ on both sides of the above equation we can factor out $\int_{-\eps}^\eps \mu\in \CC^N$ on the left-hand side. With the definition
\begin{align*}
	Q(\delta k)\DEF \frac{1}{(\delta k)^2}|D|^{-1}+\frac{2\log(\eps/2)}{\pi} +\,R^{\delta k}\,,
\end{align*}
where $\log(\eps/2)$ is a $N\times N$ diagonal matrix with entries $\log(\eps_i/2)$, we can solve for $\int_{-\eps}^\eps \mu$ and obtain
\begin{align*}
	\int_{-\eps}^\eps \mu 
		\;=\; Q(\delta k)^{-1}\,\fdk\,.
\end{align*}
Embedding this in Equation (\ref{equ:Opening_Id_LLetilde}), we obtain
\begin{align*}
	\del_\nu \udk(t) =
		&\frac{\pi}{2} (\LLe)^{-1} \Big[
			\fdk 
			- (\tfrac{1}{(\delta k)^2}|D|^{-1}+R)\int_{-\eps}^\eps \mu
		\Big]\,
		\\
		=
		&\frac{ \Big(
			\mathrm{I}_N-
			\big(\tfrac{1}{(\delta k)^2}|D|^{-1}+ R^{\delta k}\big)
			\,Q(\delta k)^{-1}
		\Big)\fdk}
		{2 \log(\nicefrac{\eps}{2})\sqrt{\eps^2-t^2}}\,
		\\
		=
		&\frac{
			\,Q(\delta k)^{-1}\,\fdk}
		{\pi \sqrt{\eps^2-t^2}}\,.
\end{align*}

Considering the error term, which we had for $\curlRdk$ as well as for the right-hand side for the equation in Lemma \ref{lemma:Gap-Formula}, we can readily establish an error term for $\del_\nu\udk$ of order $\OO\big(N\,\tfrac{\epsilon}{\sqrt{\eps^2-t^2}}\,\|Q(\delta k)^{-1}\|_2\big)$, where we used \cite[Lemma 3.14]{FWMSP1}.

\begin{remark}
	The values $\delta k_{\mathrm{res}}\in\CC$ for which the operator in Equation (\ref{equ:Opening_Id_LLetildeFULL}) is not invertible are resonance values. Using the generalized Rouché theorem \cite[Theorem 1.2]{LPTSA} we see that those resonance values are of order $\max_i(\eps_i)$ away from those values $\delta k_{\mathrm{app}}\in\CC$ in Equation (\ref{equ:Opening_Id_LLetilde}), which yield no solution $\mu$. We can readily give a condition to determine the $\delta k_{\mathrm{app}}$, that is,
	\begin{align*}
		0\in \mathrm{eig}(Q(\delta k))\,,
	\end{align*}	 
	where $\mathrm{eig}(M)$ is the set of all eigenvalues of a matrix $M\in\CC^{N\times N}$. 
\end{remark}

\begin{remark}
	In \cite[Section 4.3.4]{FWMSP1} it is shown that for $\frac{1}{|\log(\eps_i/2)|}$ small enough we have exactly $2 N$ eigenvalues. This is done using the generalized Rouché theorem \cite[Theorem 1.2]{LPTSA} and the generalized argument principle \cite[Theorem 1.12]{LPTSA}. Note that $\eps_i$ has to be too small for those findings to hold.
\end{remark}

\subsection{Solving for $\udk$ on the far-field}
Consider that for $z\in Y$, but $z$ not element of the closure of our Helmholtz resonators, it holds analogous to the proof of Lemma \ref{lemma:Gap-Formula}, using Green's identity, the Dirichlet condition, the quasi-periodicity condition and the radiation condition, 
\begin{align}\label{equ:udk on closefield}
	\udk (z) = \sum_{i=1}^N \int_{\Lambda_i} \Ndkp(z,y)\del_\nu \udk(y)\intd\sigma_y
		  + 2\imagi\delta k_2\int_{-\nicefrac{1}{2}}^{\nicefrac{1}{2}} \Ndkp(z,{\smallvec{y_1\\r}}) u_0^{\delta k}({\smallvec{y_1\\r}})\intd y_1 \,
\end{align}
for $r\rightarrow\infty$. We recall that $\Ndkp(z,x)=\Gdkp(z,x)+\Rdkp(z,x)$.
Using Equation (\ref{equ:ri}), that is, 
$\Rdkp(z,x) = \euler^{-\imagi\delta k_1 z_1}\, \euler^{\imagi\delta k_2 z_2}\, \euler^{\imagi\delta k_1 x_1}\;\ri+\OO(\euler^{- C\,z_2})$ 
for $z_2\rightarrow\infty$ and $x_2=h_i$, where $\ri$ does not depend on $z$, and using Equation (\ref{equ:rex}), that is, 
$\Rdkp(z,x) =\euler^{-\imagi\delta k_1 z_1}\, \euler^{\imagi \delta k_1 x_1}\,\euler^{\imagi \delta k_2 z_2}\, \euler^{\imagi\delta k_2 x_2}\,\rex +\OO(\euler^{- C\,(x_2-z_2)})$
for $x_2\rightarrow\infty$, where $\rex$ does not depend on $z$ and $x$,
we obtain that	
\begin{align*}
	\udk (z) = 
		\euler^{-\imagi\delta k_1 z_1}\, 
		\euler^{\imagi\delta k_2 z_2}
		\,\Big(
			&\sum_{i=1}^N 
				\big(Q(\delta k)^{-1}\,\fdk\big)_i\,\euler^{\imagi\delta k_1 \xi_i}
				\Big(-\frac{\sin(\delta k_2\,h_i)}{\delta k_2}+\ri\Big)\\
		&- I_0(1-2\imagi\,\delta k_2\,\rex)
		\Big) \;+\; \udk_0(z)\,,
\end{align*}
with an error in $\OO(N\,\eps\,\|Q(\delta k)^{-1}\|_2)+\OO(\euler^{-C\,z_2})$. Re-establishing the macroscopic view and recovering the notation for $Y, D_i, \Lambda_i, l_i, h_i, \xi_i, \eps$, we proved the main result.

\newcommand{\raystretch}[1]{\renewcommand{\arraystretch}{#1}}

\section{Numerical Implementation and Application of the Main Results}\label{sec:NumImplTest}

\subsection{Numerical Implementation}\label{sec:NumImplementation}
To carry out our numerics we rely on Matlab. To find the eigenvalues we use the in-built function \verb+eig+. To find the constant matrix $R^{\delta k}\in\CC^{N\times N}$ and constant vectors $\ri\,,\ridel\in\CC$, for $i\in\{1,\ldots N\}$, and the constant function $\rex\in\CC$, we have to implement $\Gkp, \Nkidel, \Nkp, \Nkpdel$ using integral equations, which we will explain in the following. 

Unfortunately, the series expansion for $\Gkp$ given in Equation (\ref{equ:GammaPlus}) is not converging fast enough and thus we have to apply a fast converging alternative representation, that is the Ewald method, which is described in \cite{EwaldMethod} and which is slightly altered to fit our definition of $\Gkp$ in \cite[Section 4.1]{BubbleMetaScreen}. Here we have to mention that for $k\approx 2\pi\,n-|k_1|$, for any $n\in \NN$, the implementation of $\Gkp$ does not well-behave because of an instability phenomenon, which is called the case of empty resonance. To resolve this one can use different implementations like the Barnett-Greengard method \cite[Section 5.4]{MCMP}.

The computation of $\Nkidel$, is a well-studied process, see, for instance,  \cite[Proposition 2.8 ]{LPTSA} and \cite{NeumannFunctionAdd}. For our implementation, we used the representation $\Nkidel(z,x)=2\,\Gamma^{k}(z,x)+\mathrm{R}^{k}(z,x)$, where $\Gamma^{k}(z,x)=-\frac{\imagi}{4}H_{(0)}^{(1)}(k\,|z-x|)$, where $H_{(0)}^{(1)}$ is the Hankel function of the first kind of order 0, and $\mathrm{R}^{k}$ satisfies the Helmholtz equation with the boundary condition $\del_{\nu_x}\mathrm{R}^{k}(z,x)=-2\,\del_{\nu_x}\Gamma^{k}(z,x)$. Using Green's identity we obtain that
\begin{align*}
	\mathrm{R}^{k}(z,x) -2\int_{\del D_i}\mathrm{R}^{k}(z,y)\del_{\nu_y}\Gamma^k(x,y)\intd\sigma_y
		= -4\int_{\del D_i}\Gamma^k(x,y)\del_{\nu_x}\Gamma^{k}(z,y)\intd\sigma_y\,.
\end{align*}
Applying a discretization with $M$ nodes to the boundary, and approximating the integral with the trapezoidal rule, we obtain a linear system of equation of the form $(\,\frac{1}{2}\,\mathrm{I}_M-{\verb+K+}\,) \,{\verb+R+}=\verb+f+$, from which we obtain a discretized version \verb+R+ of $\mathrm{R}^{k}(z,\cdot)$. In this process, the problem arises that the trapezoidal rule cannot efficiently approximate singular integrals, that is with increasing number of nodes $M$ the error in the approximation might not decay. These troublesome integrals are $\int_{-\eps}^\eps \phi(t)\log(|t|)\intd t$ and $\int_{0}^T \phi(t)\frac{1/M}{t^2+(1/M)^2}\intd t$, for a smooth enough $\phi\in L^2((-\eps_1,\eps_2))$. To overcome the inaccurate approximation we use the following identities:
\begin{align*}
	\int_{-\eps_1}^{\eps_2} \phi(t)\log(|t|)\intd t
		&= \Big[\log(|t|)\int_0^t\phi(\tau)\intd\tau\Big]_{-\eps_1}^{\eps_2}-\int_{-\eps_1}^{\eps_2}\frac{1}{t}\int_0^t\phi(\tau)\intd\tau\intd t\,,\\
	\int_{0}^T \phi(t)\frac{1/M}{t^2+(1/M)^2}\intd t
		&= \int_0^{M\,T}\phi(\arctan(M\,t))\intd t\,,
\end{align*}
and subsequently we use Taylor series for $\phi$ in $\int_0^t\phi(\tau)\intd\tau$ and we condense the nodes close to the bounds of the interval for $\int_0^{M\,T}\phi(\arctan(M\,t))\intd t$. 

Given now $\mathrm{R}^k$, we obtain $\Rkidel$ by
\begin{align*}
	\Rkidel(z,y)= \mathrm{R}^k(z,y)-\frac{1}{k^2\,|D_i|} + 2\big(\Gamma^k(z,x)-\tfrac{1}{2\pi}\log(|z-x|)\big)\,.
\end{align*}
This has the disadvantage that for $k\rightarrow 0$, $\Rkidel$ should converge to a finite value, but because the implementation of $\mathrm{R}^k$ yields a low accuracy for $k\rightarrow 0$, we have that the above implementation of $\Rkidel$ still diverges for $k\rightarrow 0$. To approximate $\Rkidel$ for $k=0$, we choose some non-zero values $k$, for which $\Rkidel$ is stable and extrapolated with a second degree polynomial in $k$. After some testing, we found out that $\min(\frac{\pi}{l_i},\frac{\pi}{h_i})\times [0.25,0.75]$ is an interval of well-computable values for $k$.

The computation of $\Nkpdel$, follows the same idea as for $\Nkidel$, that is we use Green's identity with $\Gkp$ and obtain an integral equation for $\Rkpdel$, which is Equation (\ref{equ:integralEqu-Rkpdel}), if $z$ and $x$ are on the boundary. Applying a discretization with $M$ nodes to the boundary, and approximating the integral with the trapezoidal rule, this once more leads to a linear system of equation of the from $(\,\frac{1}{2}\,\mathrm{I}_M+{\verb+K+}\,) \,{\verb+R+}=\verb+f+$, with the discretized version \verb+R+ of $\mathrm{R}^{k}_{+,\del}(z,\cdot)$. Fortunately, $\Gkp$ exhibits no singular behaviour in $k$ and moreover it has the same spatial singularity behaviour as $\Gamma^k$.

If $x$ is outside of the boundary, we apply a discretization of the integral equation (\ref{equ:integralEqu-Rkpdel:xoutside}) which requires the solution to the integral equation (\ref{equ:integralEqu-Rkpdel}). If $z$ is outside of the boundary we proceed analogously, by first solving a discretized version of Equation (\ref{equ:integralEqu-Rkp}) and then applying it to Equation (\ref{equ:integralEqu-Rkp:xoutside}) if $x$ is also outside the boundaries.

Then the computation of the matrix $R^{\delta k}$, the vectors $\ri$ and $\ridel$ and the constant $\rex$ can be done directly by applying Lemma \ref{lemma:rdel,r,rex}.

We tested our implementation using a discretized version of the Helm\-holtz equation and checking whether $\Rkpdel$ and $\Rkp$ satisfy those. We could not check whether those two functions satisfy the boundary condition because the numerical implementation is unstable for values $z$ too close to the boundary. Another way to test our implementation was to double the periodicity, that is if only one Helmholtz resonator was in the unit strip, we should get the results when we make the unit strip twice as wide. We got in all cases a very small error, which we suspect to be due to the discretization.

\subsection{Numerical Applications}\label{sec:NumApp}

\begin{wrapfigure}{r}{0.4\textwidth} 
  \centering
  \includegraphics[width=0.4\textwidth]{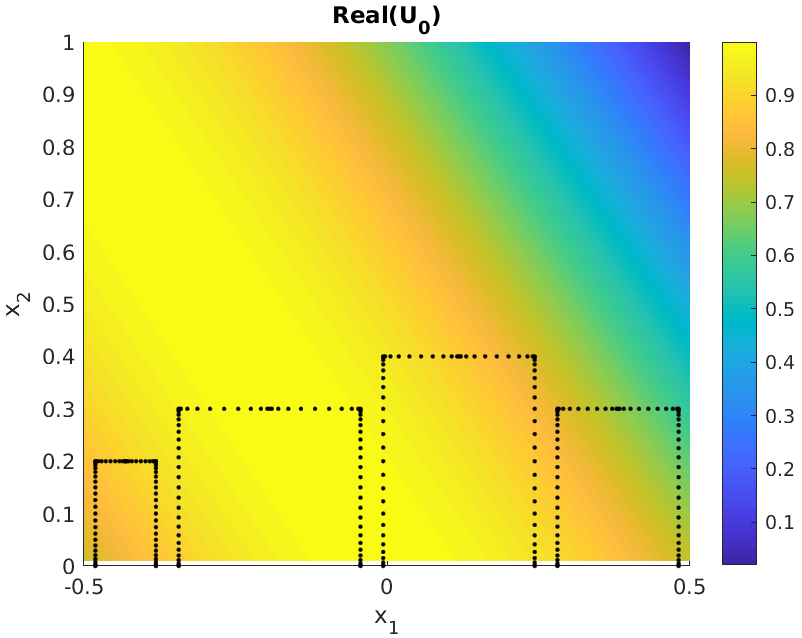}
  \caption{The four Helmholtz resonators of our unit strip are shown. The incident wave has the wave vector $(-k \cos(\theta),-k \sin(\theta))$, where $k =1.663$ and $\theta=\pi/6$.\vspace{-20pt}}\label{fig:U0}
\end{wrapfigure}

Our geometry involves four Helmholtz resonators. The resonator parameters are given by the following table:

\resizebox{0.5\columnwidth}{!}{
\begin{tabular}{c c c c}
	$h_1=0.2$ & $h_2=0.3$ & $h_3=0.4$ & $h_4=0.3$\\
	$l_1=0.1$ & $l_2=0.3$ & $l_3=0.25$ & $l_4=0.2$\\
	$\xi_1=-0.43$ & $\xi_2=-0.19$ & $\xi_3=0.11$ & $\xi_4=0.38$\\
\end{tabular}
}

We set $\delta$ to be $1$, since it is only a scaling parameter. 

In Figure \ref{fig:U0} we see an illustration of the unit strip $Y$ with an incoming wave with $k=1.663$ and a wavelength of $\frac{2\pi}{k}\approx 3.778$. The 300 black points on the boundary of our Helmholtz resonators are the discretization points used in the numerics.

With these values we can already determine the matrix $R^{k}\in\CC^{4 \times 4}$ in Theorem \ref{thm:mainresults}. We are especially interested in the dependence of its eigenvalues on $k \in (0,\pi)$ given the  angle of incidence $\theta$. Additionally, we want to show the behaviour of the eigenvalues of $R^{k}$, when $\theta$ varies. In Figure \ref{fig:Reigv} we have depicted for five different angles of incidence $\theta$ the real parts of the four eigenvalues of $R^{k}$ as a function of $k \in (1.5, 1.755)$. 

For $\theta=\frac{\pi}{2}$ (normal incidence to the meta-surface), we have one non-horizontal  line and three horizontal ones corresponding to eigenvalues with constant real parts in terms of $k$.  For each of these three eigenvalues, the imaginary part is much smaller than the real part while the eigenvalue associated with the non-horizontal line has an imaginary part and a real part  of the same order.

%

\begin{figure}[h]
  \centering
  \includegraphics[width=0.99\textwidth]{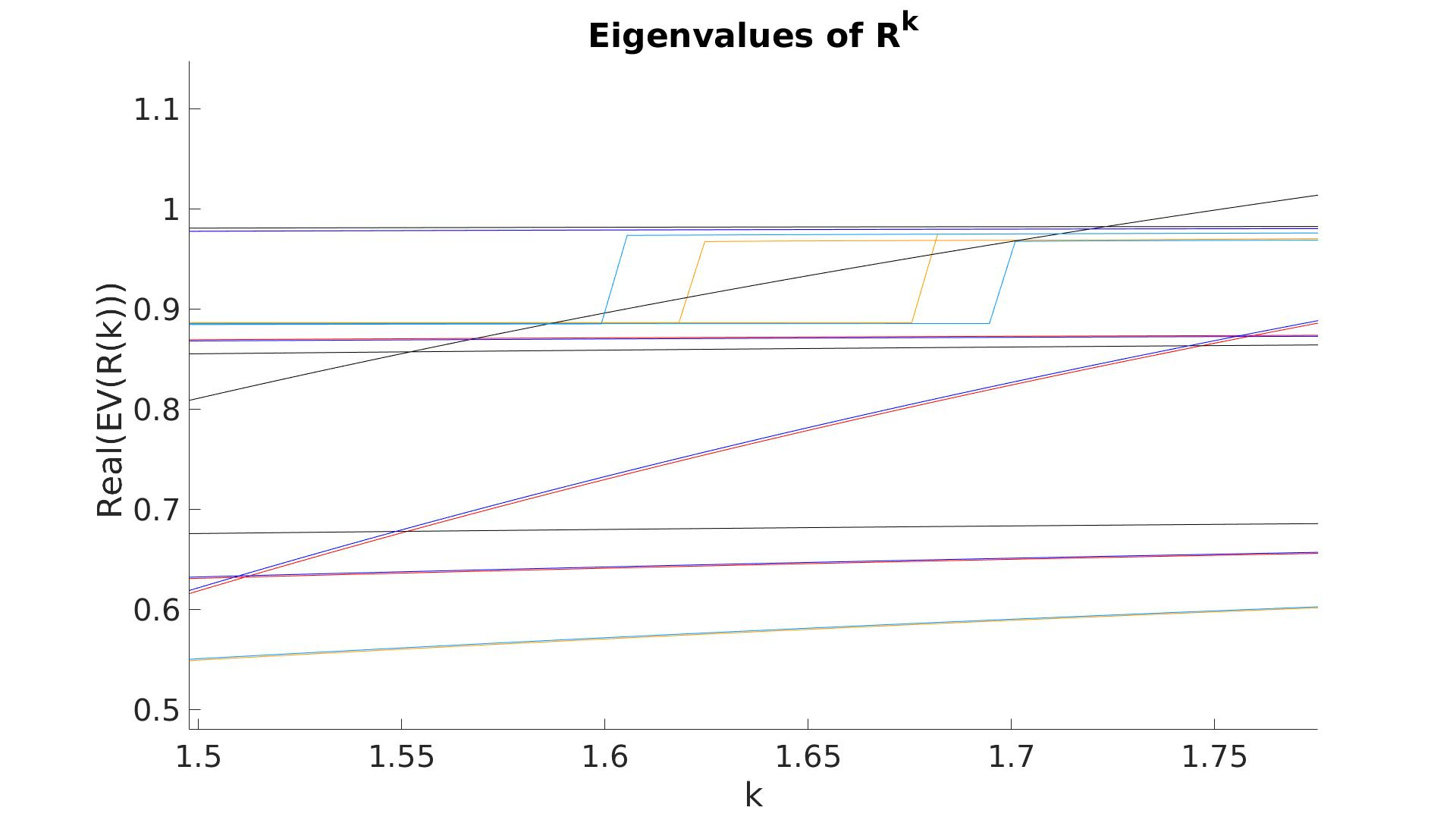}
  \caption{Real parts of the four eigenvalues of $\mathrm{R}^k$ for $k\in [1.5,  1.775]$, corresponding to five different angles of incidence $\theta$ (in green for $\theta=\frac{5}{6}\pi$, in blue for $\theta=\frac{2}{3}\pi$, in black for $\theta= \pi$, in red for $\theta=\frac{1}{3}\pi$, and in yellow for $\theta=\frac{1}{6}\pi$). Here, $EV(R(k))$ stands for eigenvalue of $\mathrm{R}^k$. }\label{fig:Reigv}
\end{figure}

In Figure \ref{fig:Reigv}, we pick the parameter $k\approx 1.663$ and $0.8858 \in \mathrm{Real}(\mathrm{eig}(\mathrm{R}^k))$ in order to determine the $\eps_i$ for $i=1, \ldots,4$. As elaborated after Theorem \ref{thm:mainresults}, we take for $\eps_1$ to $\eps_4$ the values
\begin{align*}
	\eps_i = 2\,\exp\big(\!\!-\tfrac{\pi}{2}\big(\tfrac{1}{|D_i|(1.663)^2}+0.8858\big)\big), \quad i=1,2,3,4.
\end{align*}
We note here that not all $\eps_i$ are physically meaningful, for example we have that $\eps_1 \approx 10^{-14}$, but to give an understanding of the concept, we allow $\eps_i$ to hold these values.

Now that we have $\eps_1$ to $\eps_4$ we can compute $I_s$. Since $I_s$ is a complex number, we can decompose it into a phase $\theta_{I_s}\in(-\pi,\pi]$ and an amplitude  $|I_s|>0$, as $I_s=|I_s|\exp(\imagi\theta_{I_s})$. In Figure \ref{fig:Is} we show the result. We see that for $k\approx 1.663$, that is a wavelength of approximately $3.778$, and for the angle of incidence $\theta \approx\pi$ we have almost no absorption, that is $|I_s|\approx 1$. And just below that value for $k$ we see a full absorption, that is $|I_s|\approx 0$. In Figure \ref{fig:Is} (left), we see an abrupt shift of $\theta_{I_s}$ in that neighbourhood of $k$.

Here we want to note that in Figure \ref{fig:Reigv} we have a second accumulation of the real parts of the eigenvalues at $(1.65,0.98)$. For those values, we achieve a singular behaviour for $Q(k)$ too, but the values of the vector $Q(k)^{-1}f^k$ are moderate, and we do not see a strong resonance effect.

\begin{figure}[h]
  \centering
  \includegraphics[width=0.99\textwidth]{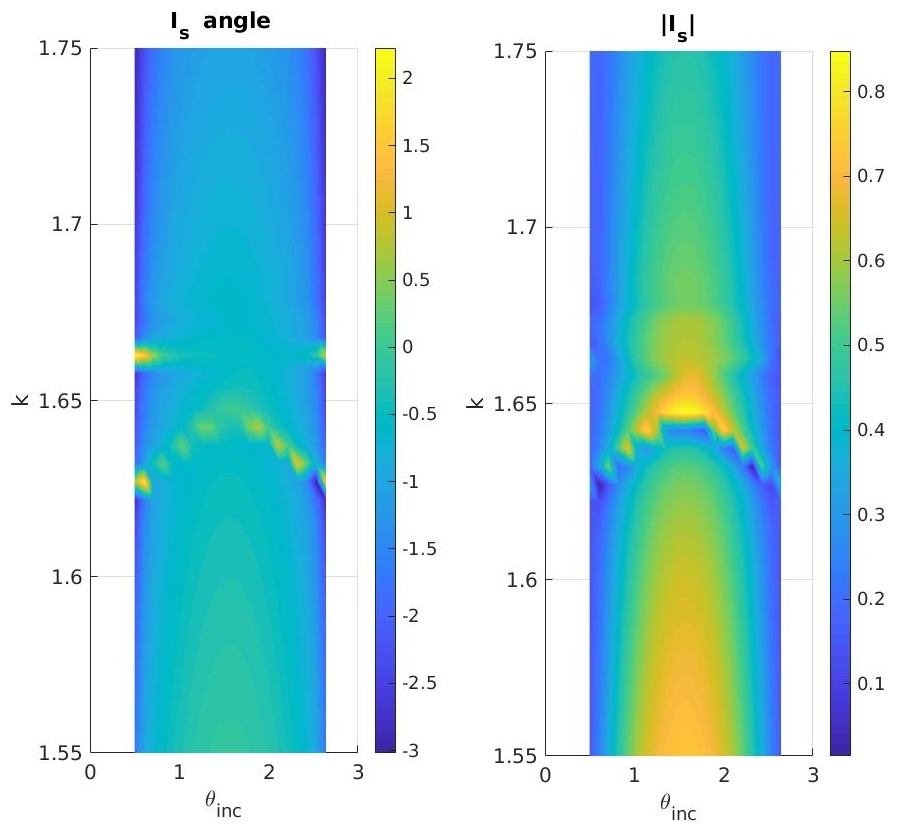}
  \caption{Phase and amplitude of $I_s$ for $k\in(1.5, 1.755)$ and $\theta \in (0.5,2.5)$.}\label{fig:Is}
\end{figure}

From the proof of Theorem \ref{thm:mainresults}, specifically from Equation (\ref{equ:udk on closefield}), we can recover an asymptotic formula for the resulting wave $U^k(z)$ for $z$ close to the Helmholtz resonators, which is given by
\begin{align*}
	U^k(z)= 
		&\sum_{i=1}^N\, \Ndkp\big(\tfrac{z}{\delta} ,{\smallvec{\xi_i\\h_i}}\big) \;\big(Q(\delta k)^{-1} f^{\delta k}\big)_i
		+2\imagi\, I_0 \,\delta k_2 \,\lim_{r\rightarrow\infty}\frac{\Rdkp\big(\tfrac{z}{\delta} ,{\smallvec{0\\r}}\big)}{\euler^{+\imagi \delta k_2 r}}\\
		&+\;U_0^k(z)\;-\;I_0\, \euler^{-\imagi k_1 z_1} \euler^{-\imagi k_2 z_2} \;+\; \OO(N\,\eps\,\|Q(\delta k)^{-1}\|_2)\,,
\end{align*}
where $Q(\delta k)$ and $\fdk$ are as in Theorem \ref{thm:mainresults}. The field inside a Helmholtz resonator is given through Equation (\ref{equ:u inside D}). For $k=1.663$ and $\theta=\pi/6$, we obtain the solution $\mathrm{Real}(U^k(z))$ shown in Figure \ref{fig:Ufield}, where $z$ is located close to the Helmholtz resonator $D_1$ to $D_4$. We also see focal spots with $|\mathrm{Real}(U^k(z))|>1$ on the left and right sides of the unit strip. This outcome is similar to the findings in \cite{HRSuperLens}.

\begin{figure}[h]
  \centering
  \includegraphics[width=0.99\textwidth]{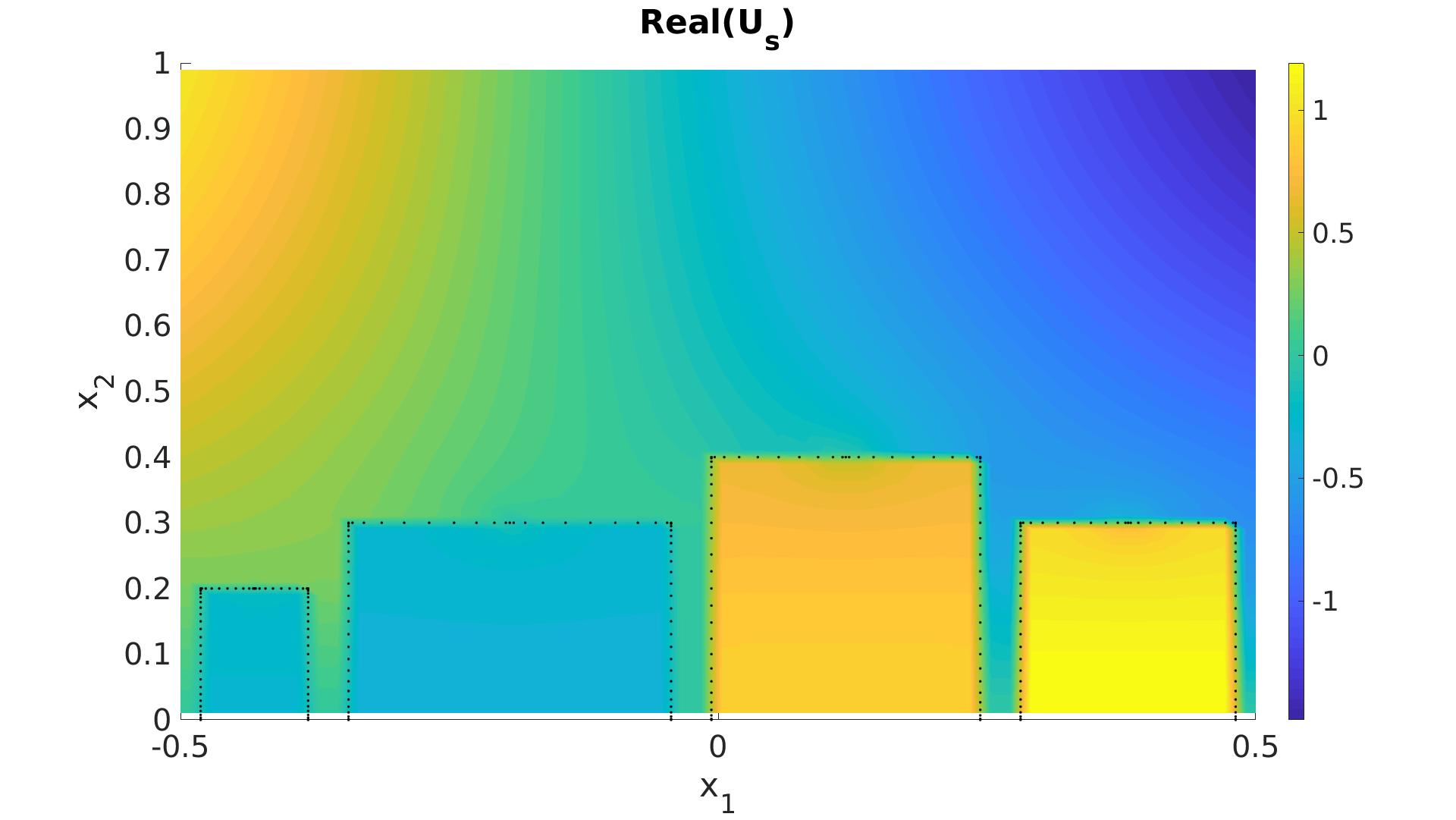}
  \caption{The field $\mathrm{Real}(U^k(z))$ is depicted, where $z$ is close enough to the Helmholtz resonators, $k=1.663$ and $\theta=\pi/6$. This shows the resulting scattering when applying the same set-up as the one in Figure \ref{fig:U0}.}\label{fig:Ufield}
\end{figure}

Finally, it is worth noticing that after simulating several generic examples, it appears that in a set-up with $N$ Helmholtz resonators, we have $N-1$ resonances, which exhibit a small imaginary part and remain at given  $\theta$ constant in terms of  $k$, and one resonance, which has a large imaginary part and whose real part at given $\theta$ increases with  $k$.

\section{Concluding Remarks}\label{sec:ConcludingRemarks}
In this paper we have established in Theorem \ref{thm:mainresults} that the formula $U^k(z)-U_0^k(z) = I_s\;\euler^{-\imagi k_1 z_1}\, \euler^{\imagi k_2 z_2}$ describes the scattered field in the far-field in our geometry, up to a small error, which depends on the lengths of the openings of the Helmholtz resonators. Moreover, we have seen that $I_s=(v^{\delta k})^\TransT\,Q(\delta k)^{-1}\,f^{\delta k}-b$, for $Q(\delta k)\in\CC^{N\times N}$ and $f^{\delta k}\in\CC^N$ as in Theorem \ref{thm:mainresults}, for some vector $v^{\delta k}\in\CC^N$ and some scalar $b\in\CC$, where $Q(\delta k)$ is the addition of a diagonal matrix $D^{\eps,\delta k}$ and the matrix $R^{\delta k}$. $D^{\eps,\delta k}$ is given through a simple formula consisting of all $\eps_i$ and  $\delta k$. $R^{\delta k}$ contains the intrinsic properties of the geometry. By determining the diagonal matrix according to the eigenvalues of $R^{\delta k}$ we made the matrix $Q(\delta k)$ nearly singular for some  values of $\delta k$, which in turn leads to an unusual behaviour for the intensity $I_s$ of the scattered field due to  the resonances of the resonators. In the numerical applications we have achieved almost full absorption and almost no absorption for wavelengths of around three units and additionally, we have obtained an abrupt phase-shift in that regime. These effects are similar to those observed experimentally for gradient meta-surfaces in \cite{add,BroadbandAnomalousReflection}. 

The proof requires that the wave number $\delta k$ satisfies $\delta k<2\pi-|\delta k_1|$. In numerical simulations, this constraint leads to $\delta k \in [0, 2.5]$.  If we look for shorter wavelengths, we need to rewrite $\Gdkp$, in Equation (\ref{equ:GammaPlus}), and Lemma \ref{lemma:rdel,r,rex}. This will make the formulas more complicated, but might also yield new and unforeseen results.

Since the connection between the parameters of the geometry ($\xi_i, l_i$ and $h_i$) and the eigenvalues of the matrix $R^{\delta k}$ is not explicit, the only way to evaluate $R^{\delta k}$ is through numerically solving integral equations. Hence we are left with trying different geometries until finding a geometry with the desired properties. One future point of interest in order to find optimal configurations might be to perform a sensitivity analysis, when we change either $\xi_i, l_i$ or $h_i$. We think that the case when two Helmholtz resonators get close to touching might reveal fascinating phenomena, as it does in the case of two nearly touching plasmonic spheres, see \cite{SanghyeonHabibSpheres, pnas}. 

Our framework is established in two dimensions. As said in the introduction, the authors of \cite{BroadbandAnomalousReflection} saw in their physical experiments no influence of the dimension of depth, that is the third dimension. However, a physical application of our gradient meta-surface might require a formula derived from a three dimensional Helmholtz resonator. To gather such a formula one needs a higher dimensional Green's function $\Gkp$, that is given in \cite{BubbleMetaScreen}, and associated higher dimensional hyper-singular integral operators. An application of those is given in \cite{HabibHai}.

In our mathematical model we used the Dirichlet and Neumann boundary conditions. Materials which satisfy these conditions perfectly are not commonly available. Thus different boundary conditions might model the physical problem more accurately. For example in \cite{oursteklovpaper}, Steklov boundary conditions are studied. Moreover, mixed boundaries as they are studied in \cite{ourzarembapaper} move the singularity in $k$ of the scattered wave away from $k=0$. Thus with mixed boundaries we might achieve a better control of the resonance effect. Additionally, our model has boundaries of vanishing width. To implement non-vanishing widths, one has to derive the behaviour of the scattered wave on the neck of the Helmholtz resonators.

Our results can lead to many practical applications. We are looking forward to experimental realizations of the unusual effects
of gradient meta-surfaces highlighted in this paper.


\bibliographystyle{plain}
\bibliography{ms}

\newcommand{\noop}[1]{} \def\cprime{$'$}
\begin{thebibliography}{10}

\bibitem{ourzarembapaper}
H.~{Ammari}, O.~Bruno, K.~{Imeri}, and N.~{Nigam}.
\newblock {Wave enhancement through optimization of boundary conditions}.
\newblock {\em SIAM J. Sci. Compt., to appear
  (https://doi.org/10.1137/19M1274651)}, 2020.

\bibitem{oursteklovpaper}
H.~{Ammari}, K.~{Imeri}, and N.~{Nigam}.
\newblock {Optimization of Steklov-Neumann eigenvalues}.
\newblock {\em J. Compt. Phys.}, 406:109211, 2020.

\bibitem{LPTSA}
H.~Ammari, H.~Kang, and H.~Lee.
\newblock {\em Layer potential techniques in spectral analysis}, volume 153 of
  {\em Mathematical Surveys and Monographs}.
\newblock American Mathematical Society, Providence, RI, 2009.

\bibitem{BubbleMetaScreen}
Habib Ammari, Brian Fitzpatrick, David Gontier, Hyundae Lee, and Hai Zhang.
\newblock A mathematical and numerical framework for bubble meta-screens.
\newblock {\em SIAM J. Appl. Math.}, 77(5):1827--1850, 2017.

\bibitem{MCMP}
Habib Ammari, Brian Fitzpatrick, Hyeonbae Kang, Matias Ruiz, Sanghyeon Yu, and
  Hai Zhang.
\newblock {\em Mathematical and computational methods in photonics and
  phononics}, volume 235 of {\em Mathematical Surveys and Monographs}.
\newblock American Mathematical Society, Providence, RI, 2018.

\bibitem{FWMSP1}
Habib Ammari, Kthim Imeri, and Wei Wu.
\newblock A mathematical framework for tunable metasurfaces. {P}art {I}.
\newblock {\em Asymptot. Anal.}, 114(3-4):129--179, 2019.

\bibitem{NeumannFunctionAdd}
Habib Ammari, Hyeonbae Kang, and Mikyoung Lim.
\newblock Gradient estimates for solutions to the conductivity problem.
\newblock {\em Mathematische Annalen}, 332(2):277--286, 2005.

\bibitem{HabibHai}
Habib Ammari and Hai Zhang.
\newblock A mathematical theory of super-resolution by using a system of
  sub-wavelength {H}elmholtz resonators.
\newblock {\em Comm. Math. Phys.}, 337(1):379--428, 2015.

\bibitem{Bao95}
Gang Bao, David~C. Dobson, and J.~Allen Cox.
\newblock Mathematical studies in rigorous grating theory.
\newblock {\em J. Opt. Soc. Am. A}, 12(5):1029--1042, May 1995.

\bibitem{EwaldMethod}
F.~{Capolino}, D.~R. {Wilton}, and W.~A. {Johnson}.
\newblock Efficient computation of the 2d green's function for 1d periodic
  layered structures using the ewald method.
\newblock In {\em IEEE Antennas and Propagation Society International Symposium
  (IEEE Cat. No.02CH37313)}, volume~1, pages 194--197 vol.1, June 2002.

\bibitem{reviewmeta}
F.~{Ding}, A.~{Pors}, and S.I. {Bozhevolnyi}.
\newblock {Gradient metasurfaces: a review of fundamentals and applications}.
\newblock {\em Rep. Prog. Phys.}, 81:026401, 2018.

\bibitem{LaplEigenValuesFunctions}
D.~S. Grebenkov and B.-T. Nguyen.
\newblock Geometrical structure of {L}aplacian eigenfunctions.
\newblock {\em SIAM Rev.}, 55(4):601--667, 2013.

\bibitem{Lingling13}
{Huang Lingling}, {Chen Xianzhong}, {Mühlenbernd Holger}, {Zhang Hao}, {Chen
  Shumei}, {Bai Benfeng}, {Tan Qiaofeng}, {Jin Guofan}, {Cheah Kok-Wai}, {Qiu
  Cheng-Wei}, {Li Jensen}, {Zentgraf Thomas}, and {Zhang Shuang}.
\newblock Three-dimensional optical holography using a plasmonic metasurface.
\newblock {\em Nature Communications}, 4(1):2808, 2013.

\bibitem{HRSuperLens}
{Lan Jun}, {Li Yifeng}, {Xu Yue}, and {Liu Xiaozhou}.
\newblock Manipulation of acoustic wavefront by gradient metasurface based on
  helmholtz resonators.
\newblock {\em Scientific Reports}, 7(1):10587, 2017.

\bibitem{Lin298}
Dianmin Lin, Pengyu Fan, Erez Hasman, and Mark~L. Brongersma.
\newblock Dielectric gradient metasurface optical elements.
\newblock {\em Science}, 345(6194):298--302, 2014.

\bibitem{hai5}
Junshan Lin, Stephen~P. Shipman, and Hai Zhang.
\newblock A mathematical theory for fano resonance in a periodic array of
  narrow slits.
\newblock {\em arXiv:1904.11019}.

\bibitem{hai6}
Junshan Lin and Hai Zhang.
\newblock Fano resonance in metallic grating via strongly coupled subwavelength
  resonators.
\newblock {\em arXiv:1911.01025}.

\bibitem{hai4}
Junshan Lin and Hai Zhang.
\newblock Scattering by a periodic array of subwavelength slits i: field
  enhancement in the diffraction regime.
\newblock {\em Multiscale Model. Simul.}, 16(2):922–953, 2018.

\bibitem{hai3}
Junshan Lin and Hai Zhang.
\newblock Scattering by a periodic array of subwavelength slits ii: surface
  bound states, total transmission, and field enhancement in homogenization
  regimes.
\newblock {\em Multiscale Model. Simul.}, 16(2):954–990, 2018.

\bibitem{hai2}
Junshan Lin and Hai Zhang.
\newblock An integral equation method for numerical computation of scattering
  resonances in a narrow metallic slit.
\newblock {\em J. Comput. Phys.}, 385(6):75–105, 2019.

\bibitem{hai1}
Junshan Lin and Hai Zhang.
\newblock Mathematical analysis of surface plasmon resonance by a nano-gap in
  the plasmonic metal.
\newblock {\em SIAM J. Math. Anal.}, 51(6):4448–4489, 2019.

\bibitem{maurel1}
A.~Maurel, J.-J. Marigo, J.-F. Mercier, and K.~Pham.
\newblock Modelling resonant arrays of the helmholtz type in the time domain.
\newblock {\em Proc. R. Soc. A}, 474:20170894, 2018.

\bibitem{maurel2}
A.~Maurel, J.-F. Mercier, K.~Pham, J.-J. Marigo, and A.~Ourir.
\newblock Enhanced resonance of sparse arrays of helmholtz
  resonators--application to perfect absorption.
\newblock {\em The Journal of the Acoustical Society of America},
  145(4):2552--2560, 2019.

\bibitem{Xingjie13}
{Ni Xingjie}, {Kildishev Alexander V.}, and {Shalaev Vladimir M.}
\newblock Metasurface holograms for visible light.
\newblock {\em Nature Communications}, 4(1):2807, 2013.

\bibitem{FlocquetTheory}
Michael Reed and Barry Simon.
\newblock {\em Methods of modern mathematical physics. {IV}. {A}nalysis of
  operators}.
\newblock Academic Press [Harcourt Brace Jovanovich, Publishers], New
  York-London, 1978.

\bibitem{SaranenVainikko2002}
J.~Saranen and G.~Vainikko.
\newblock {\em Periodic integral and pseudodifferential equations with
  numerical approximation}.
\newblock Springer Monographs in Mathematics. Springer-Verlag, Berlin, 2002.

\bibitem{BroadbandAnomalousReflection}
Shulin Sun, Kuang-Yu Yang, Chih-Ming Wang, Ta-Ko Juan, Wei~Ting Chen, Chun~Yen
  Liao, Qiong He, Shiyi Xiao, Wen-Ting Kung, Guang-Yu Guo, Lei Zhou, and
  Din~Ping Tsai.
\newblock High-efficiency broadband anomalous reflection by gradient
  meta-surfaces.
\newblock {\em Nano Letters}, 12(12):6223--6229, 2012.
\newblock PMID: 23189928.

\bibitem{add}
N.~Yu, P.~Genevet, M.A. Kats, F.~Aieta, J.-P. Tetienne, F.~Capasso, and
  Z.~Gaburro.
\newblock Light propagation with phase discontinuities: Generalized laws of
  reflection and refraction.
\newblock {\em Science}, 334:333--337, 2011.

\bibitem{SanghyeonHabibSpheres}
S.~Yu and H.~Ammari.
\newblock Plasmonic interaction between nanospheres.
\newblock {\em SIAM Review}, 60(2):356--385, 2018.

\bibitem{pnas}
S.~Yu and H.~Ammari.
\newblock Hybridization of singular plasmons via transformation optics.
\newblock {\em Proc. Natl. Acad. Sci. USA}, 116(28):13785--13790, 2019.

\bibitem{Nanfang14}
{Yu Nanfang} and {Capasso Federico}.
\newblock Flat optics with designer metasurfaces.
\newblock {\em Nature Materials}, 13(2):139–150, 2014.

\end{thebibliography}

\end{document}